\documentclass{article}
\usepackage{harvard}
\usepackage{graphicx}
\usepackage{camnum}
\usepackage{amsmath}
\usepackage{enumerate}
\usepackage{algorithm}
\usepackage{algpseudocode}
\usepackage{tikz,pgfplots}
\usepackage{tikz-cd}

\allowdisplaybreaks

\newcommand{\ee}{\mathrm{e}}
\newcommand{\ii}{{\mathrm i}}

\newtheorem{remark}[theorem]{Remark}

\title{Sobolev-Orthogonal Systems with Tridiagonal Skew-Hermitian Differentiation Matrices}
\author{Arieh Iserles\footnote{\texttt{ai@damtp.cam.ac.uk}, Department of Applied Mathematics and Theoretical Physics, University of Cambridge, Wilberforce Road, Cambridge CB3 0WA, United Kingdom.} \and Marcus Webb\footnote{Corresponding author, \texttt{marcus.webb@manchester.ac.uk}, Department of Mathematics, University of Manchester, Alan Turing Building, Manchester M13 9PL, United Kingdom.}}

\begin{document}
\maketitle

\begin{abstract}
  We introduce and develop a theory of orthogonality with respect to Sobolev inner products on the real line for sequences of functions with a tridiagonal, skew-Hermitian differentiation matrix. While a theory of such $\CC{L}_2$-orthogonal systems is well established, Sobolev orthogonality requires new concepts and their analysis. We characterise such systems completely as appropriately weighed Fourier transforms of orthogonal polynomials and present a number of illustrative examples, inclusive of a Sobolev-orthogonal system whose leading $N$ coefficients can be computed in $\O{N\log N}$ operations.
\end{abstract}


\section{Introduction}

\subsection{Orthonormal systems on the real line}

The theory of $\CC{L}_2$-orthonormal systems on the real line with a tridiagonal differentiation matrix has been developed in \cite{iserles19oss,iserles20for,iserles21fco,Iserles21daf}. In its simplest (real) version, let $w\geq0$ be an absolutely continuous, nonzero weight function, symmetric with respect to the origin, and $\{p_n\}_{n\in\bb{Z}_+}$ the underlying system of orthonormal polynomials, which must satisfy
\begin{displaymath}
  b_n p_{n+1}(\xi)=\xi p_n(\xi)-b_{n-1}p_{n-1}(\xi),\qquad n\in\BB{Z}_+.
\end{displaymath}
for some real numbers $\{b_n\}_{n\in\bb{Z}_+}$. Setting 
\begin{equation}
  \label{eq:DefVarphi}
  \varphi_n(x)=\frac{\ii^n}{\sqrt{2\pi}} \int_{-\infty}^\infty p_n(\xi) \sqrt{w(\xi)} \ee^{\ii x\xi}\D\xi,\qquad x\in\BB{R},\quad n\in\BB{Z}_+,
\end{equation}
we obtain by Parseval's theorem an orthonormal system of functions in $\CC{L}_2(\BB{R})$. Moreover, under the mild assumption that polynomials are dense in $\CC{L}_2(\BB{R};w)$, this system is dense in $\CC{L}_2(\BB{R})$ if the support of $w$ is all of $\BB{R}$, otherwise its closure is the {\em Paley--Wiener space\/} $\mathcal{PW}_{\CC{supp}\,w}(\BB{R})$ of all $\CC{L}_2(\BB{R})$ functions whose Fourier transform is supported on $\CC{supp}\,w$. Moreover, $\Phi=\{\varphi_n\}_{n\in\bb{Z}_+}$ obeys
\begin{equation}
  \label{eq:TTR}
  \varphi_n'(x)=-b_{n-1} \varphi_{n-1}(x)+b_n\varphi_{n+1}(x),\qquad n\in\BB{Z}_+.
\end{equation}
In vector form, \R{eq:TTR} is $\MM{\varphi}'=\mathcal{D}\MM{\varphi}$, where $\mathcal{D}$ is the \emph{differentiation matrix} of the system, which in this case is tridiagonal and skew-symmetric. Skew symmetry and the tridiagonal form provide important advantages on the design of spectral methods with the basis $\Phi$ \cite{iserles19oss}. 

In this paper we generalise the theory to the case of Sobolev-orthogonal systems, where the Sobolev inner product is of the form
\begin{displaymath}
	\langle \varphi, \psi\rangle_v = \sum_{\ell=0}^\infty v_\ell \int_{-\infty}^\infty \varphi^{(\ell)}(x) \overline{\psi^{(\ell)}(x)} \,\D x,
\end{displaymath}
defined by the non-zero, non-negative sequence $\{ v_\ell\}_{\ell \in \bb{Z}_+} \subset [0,\infty)$. The $\CC{H}^s(\BB{R})$ norm, where $s \in \BB{Z}_+$ corresponds to $v_{\ell} = 1$ for $\ell = 0,1\ldots,s$ and $v_{\ell} = 0$ otherwise.

Besides the resulting theory being of interest in its own right, we can motivate our exploration in the context of spectral methods for PDEs using the example of the Ornstein--Uhlenbeck process,
\begin{equation}\label{eqn:OUPDE}
		\frac{\partial u}{\partial t}=\frac{\partial^2 u}{\partial x^2} - a \frac{\partial }{\partial x} \left(x u \right),\qquad x\in \BB{R},\quad t\geq0,
\end{equation}
with coefficient of friction described by the positive constant $a$ \cite{daprato14sei,lawler06isp}. Solutions to this PDE satisfy
\begin{equation}\label{eqn:OUdescent}
	\frac{\D}{\D t} \int_{-\infty}^\infty u_x^2(x) + u^2(x) \,\D x = - \int_{-\infty}^\infty 2 u_{xx}^2(x) + \left(2 + 3a \right) u_x^2(x) + a u^2(x) \, \D x,
\end{equation}
which shows that the solution decays monotonically to zero in the $\CC{H}^1(\BB{R})$ norm. In fact, we can drop some terms and show that $	\frac{\D}{\D t} \langle u, u \rangle_{H^1} \leq -a\langle u, u \rangle_{H^1}$, and hence the norm decreases at least exponentially with rate dependent on $a$.

Now, consider semi-discretising equation \eqref{eqn:OUPDE} in space by spectral method $u(x,t) \approx u_N(x,t) := \sum_{n=0}^N a_n(t)\varphi_n(x)$, where $\Phi = \{\varphi_n\}_{n\in\bb{Z}_+} \subset \CC{L}_2(\BB{R})$ are orthonormal with respect to the $\CC{H}^1(\BB{R})$ inner product. If a Galerkin scheme is used with respect to the $\CC{H}^1(\BB{R})$ inner product (i.e.~the residual of the PDE at each time $t$ is orthogonal to $\CC{span}\{\varphi_n\}_{n\in\bb{Z}_+}$), then the inequality \eqref{eqn:OUdescent} is also satisfied by $u_N$ (cf.~\cite[Ch.~8]{hesthaven2007spectral}). It therefore follows that any A-stable discretisation in time will be stable.

The plan of this paper is as follows. In Section~2, basing ourselves upon our earlier theory on $\CC{L}_2$ inner products, we present a complete framework for the construction of Sobolev-orthogonal systems on the real line with a tridiagonal differentiation matrix. This leads to two alternatives towards the construction of $\CC{H}^s(\BB{R})$-orthogonal systems, which are debated in Section~3: the first is the arguably more obvious approach, yet it leads to formul\ae{} which typically are impossible to express explicitly, while the second, less natural, results in a more constructive approach. Section~4 is concerned with systems based upon the familiar Hermite weight and Section~5 with bilateral (i.e., symmetrised with respect to the origin) Laguerre weights. In Section~6 we discuss Bessel-like orthogonal systems originating in various ultraspherical weights: in that case the closure of the orthogonal system is not $\CC{H}^s(\BB{R})$ but a relevant Paley--Wiener space. Section~7 generalises the discourse to non-symmetric measures. In that instance our orthogonal systems are complex-valued but the approach confers some important advantages. In particular, it allows us to generalise the Malmquist--Takenaka system to Sobolev setting while retaining the most welcome feature of this system, namely that the coefficients can be computed rapidly with Fast Fourier Transform. Finally, in Section~8 we present brief conclusions.

 \subsection{Sobolev norms beyond this paper}
 
As an aside, our original interest in orthonormal systems \R{eq:DefVarphi} has been motivated in \cite{iserles19oss} by the numerical solution of the linear Schr\"odinger equation in the semi-classical regime,
\begin{displaymath}
	\label{Schrodinger}
	\ii\varepsilon\frac{\partial u}{\partial t}=-\varepsilon^2 \frac{\partial^2 u}{\partial x^2}+V(x)u,\qquad x\in \BB{R} ,\quad t\geq0,
\end{displaymath}
given with an initial condition at $t=0, x \in \BB{R}$. Here $0<\varepsilon\ll1$, while the {\em interaction potential\/} $V$ is real. The solution of this equation conserves the standard $\CC{L}_2$ norm (which motivates the use of $\CC{L}_2$-orthogonal systems), but it also has another important invariant: its Hamiltonian, 
\begin{displaymath}
	H(u)=\int_{-\infty}^\infty [\varepsilon |u_x(x)|^2+\varepsilon^{-1}V(x)|u(x)|^2]\D x,
\end{displaymath}
is conserved. This might be viewed as a conservation of a non-standard Sobolev norm (if $V$ is positive). While the design of Hamiltonian methods for the Schr\"odinger equation is still an open problem, it motivates the work reported in this paper.

We mention in passing another example in which nonstandard Sobolev norms are non-increasing, the diffusion equation,
\begin{displaymath}
	\label{diffusion}
	\frac{\partial u}{\partial t}=\frac{\partial}{\partial x} \left[ a(x) \frac{\partial u}{\partial x}\right],\qquad x\in\BB{R},\quad t\geq0,
\end{displaymath}
where $a(x)>a_{\mathrm{min}} > 0$ for all $x \in \BB{R}$, given with an initial condition for $t=0, x \in \BB{R}$. It is readily shown that norm induced by the following nonstandard Sobolev inner product is non-increasing as a function of time,
\begin{displaymath}
	\langle u, u \rangle_a := \int_{-\infty}^\infty [a(x) u_x^2(x) + u^2(x) ] \D x.
	\end{displaymath}

We do not pursue these general Sobolev inner products in this paper, but anticipate reporting on such results in the future.

\subsection{Related work}

{\em Sobolev orthogonality:\/} Polynomials orthogonal with respect to Sobolev measures have been considered for a long while but the subject received considerable impetus with the introduction of coherent pairs in \cite{iserles91opo} and has been surveyed in \cite{marcellan91ops,marcellan15oso}. Natural questions, given the constructs \R{eq:DefVarphi} and \R{eq:TTR} are, firstly, how to generate Sobolev-orthogonal systems on the real line and, secondly, is a Fourier integral of an orthogonal polynomial system scaled by {\em any\/} reasonable function orthogonal with respect to {\em some\/} inner product, whether in a classical or Sobolev sense, in line with the $\CC{L}_2$ theory as briefly reviewed in Section~2. These related questions are the focus of this paper. Intriguingly, as things stand, the theory in this paper is heavily based on the theory of classical orthogonal polynomials (as distinct from Sobolev-orthogonal polynomials).

\vspace{6pt}
\noindent {\em Fourier--Bessel functions} \cite{diekema12diu,mantica06fbf}: Given a Borel measure $\D\mu$ and the underlying orthonormal system $\{p_n\}_{n\in\bb{Z}_+}$, we define
\begin{equation}
	\label{eq:def_varphi}
	\varphi_n(x)=\int_{-\infty}^\infty p_n(\xi) \ee^{-\ii x\xi}\D\mu(\xi)
\end{equation}
as the $n$th  {\em Fourier--Bessel function:\/} the name is motivated by the Legendre measure $\D\mu(x)=\chi_{(-1,1)}(x)\D x$, whereby $\varphi_n(x)= \sqrt{2\pi/x} \CC{J}_{n+\frac12}(x)$. Note the similarity between \R{eq:DefVarphi} and \R{eq:def_varphi} (disregarding the normalising factor and the sign in the exponential, neither of which is of much importance), namely that both are Fourier transforms of $p_n$ with added scaling function: $\sqrt{w}$ in the first instance, $w$ in the second. 

Further variation on this theme is the identity
\begin{equation}
	\label{eq:ChebyIntegral}
	\int_{-1}^1 \CC{T}_n(\xi)\ee^{\ii x\xi}\frac{\D\xi}{\sqrt{1-\xi^2}} =\pi \ii^n \CC{J}_n(x),\qquad n\in\BB{Z}_+,
\end{equation}
where $\CC{T}_n$ is the $n$th Chebyshev polynomial \cite{diekema12diu}. In that case the weight function has disappeared altogether in the Fourier integral. Note that, unlike \R{eq:DefVarphi}, Fourier--Bessel functions need not be orthogonal although, interestingly enough, disregarding signs and normalising constants, the two formul\ae{} concide (and orthogonality is recovered) for the Legendre measure. 

\setcounter{equation}{0}
\setcounter{figure}{0}
\section{Characterisation of Sobolev-orthogonal systems}\label{sec:theory}

Let us first state the desiderata. We are interested in functions $\Phi = \{\varphi_n\}_{n\in\bb{Z}_+} \subset \CC{L}_2(\BB{R})$ such that \emph{both} of the following properties hold.
\begin{enumerate}[(A)]
	\item There exists sequences $\{b_n\}_{n\in\bb{Z}_+} \subset \BB{C} \setminus\{0\}$ and $\{c_n\}_{n\in \bb{Z}_+} \subset \BB{R}$ such that
	\begin{equation}
	  \label{eq:3term}
		\varphi_n'(x) = -\overline{b_{n-1}} \varphi_{n-1}(x) + \mathrm{i} c_n \varphi_n(x) + b_n \varphi_{n+1}(x)
	\end{equation}
for $n = 0,1,\ldots$ (with $b_{-1} = 0$ by convention);
\item $\Phi$ is an orthonormal sequence with respect to the Sobolev inner product
\begin{equation}\label{eqn:Sobolevip}
	\langle \varphi, \psi\rangle_v = \sum_{\ell=0}^\infty v_\ell \int_{-\infty}^\infty \varphi^{(\ell)}(x) \overline{\psi^{(\ell)}(x)} \,\mathrm{d} x,
\end{equation}
defined by the non-zero, non-negative sequence $\{ v_\ell\}_{\ell \in \bb{Z}_+} \subset [0,\infty)$ such that $\sum_{\ell=0}^\infty v_\ell>0$.
\end{enumerate}

\begin{theorem}[\cite{iserles19oss,iserles20for}]\label{thm:vanilla}
  A sequence $\Phi = \{\varphi_n\}_{n\in\bb{Z}_+} \subset \CC{L}_2(\BB{R})$ satisfies criterion $(A)$ if and only if
  \begin{equation}\label{eqn:phinformula}
  	\varphi_n(x) = \frac{\ee^{\ii \theta_n}}{\sqrt{2\pi}} \int_{-\infty}^\infty\ee^{\ii x \xi} p_n(\xi)  g(\xi)   \, \mathrm{d}\xi,
  \end{equation}
where
\begin{itemize}
	\item $P = \{p_n\}_{n\in\bb{Z}_+}$ is an orthonormal polynomial system on the real line with respect to a probability measure on the real line with all moments finite and with infinitely many	points of increase;
	\item $\Theta = \{ \theta_n \}_{n\in\bb{Z}_+} \subset [0,2\pi)$;
	\item $g \in \CC{L}_2(\BB{R})$ satisfies $\lim_{\xi \to \pm \infty} |\xi^k g(\xi)| = 0 \text{ for } k = 0,1,2,\ldots$. We call such functions \emph{mollifiers}.
\end{itemize}
\end{theorem}
\begin{remark}\label{rem:bnpos}
	It is possible to ensure that the paramters $\{b_n\}_{n\in\bb{Z}_+}$ satisfy $b_n > 0$ without any genuine loss of generality. This is achieved by simply setting $\ee^{\ii\theta_n} = \ii^n$. We henceforth assume that $b_n > 0$.
\end{remark}
\begin{remark}\label{rem:real}
	Under the assumption of Remark \ref{rem:bnpos}, the functions $\Phi$ are real if and only if $g(\xi)$ has even real part and odd imaginary part, and $P$ is orthonormal with respect to an even measure (i.e.~it is symmetric about the axis). In this case, $b_n > 0$ and $c_n = 0$ for all $n$.
\end{remark}
Theorem \ref{thm:vanilla} and Remarks \ref{rem:bnpos} and \ref{rem:real} were proved by the present authors in \cite{iserles19oss,iserles20for} along with results characterising when such systems are orthogonal with respect to the standard inner product on $\CC{L}_2(\BB{R})$. The following Theorem generalises these orthogonality results to the Sobolev inner products in equation \eqref{eqn:Sobolevip}.
\begin{theorem}\label{thm:orthogonalPhi}
Let $\varphi$ satisfy criterion $(A)$, which implies that \eqref{eqn:phinformula} holds. Then $\varphi$ also satisfies criterion $(B)$ if and only if the mollifier $g$ satisfies
\begin{equation}\label{eqn:wtog}
 w(\xi) = v(\xi) |g(\xi)|^2,
\end{equation}
where $w(\xi)$ is the positive weight function with respect to which the polynomials $P$ are orthonormal, and $v(\xi) = \sum_{\ell=0}^\infty v_\ell \xi^{2\ell}$. In particular, it is necessary for the non-negative sequence $\{ v_\ell\}_{\ell \in \bb{Z}_+}$ to decay sufficiently fast that $v(\xi)$ is finite on the support of $w$. 
	\end{theorem}
\begin{proof}
	By Parseval's Theorem,
	\begin{displaymath}
		\int_{-\infty}^\infty \varphi_n(x)\overline{\varphi_m(x)} \,\D x = (-\ii)^{m-n} \int_{-\infty}^\infty p_n(\xi)p_m(\xi) |g(\xi)|^2 \,\D \xi.
	\end{displaymath}
Furthermore, since $\widehat{\varphi^{(\ell)}}(\xi) = (-\ii\xi)^\ell \hat{\varphi}(\xi)$ (where $\hat\varphi$ denotes the Fourier transform of $\varphi$) we have
\begin{displaymath}
	\int_{-\infty}^\infty \varphi^{(\ell)}_n(x)\overline{\varphi^{(\ell)}_m(x)} \,\D x = (-\ii)^{m-n} \int_{-\infty}^\infty p_n(\xi)p_m(\xi) \xi^{2\ell} |g(\xi)|^2 \,\D \xi.
\end{displaymath}
Therefore, 
$$
\langle \varphi_n,\varphi_m\rangle_v = (-\ii)^{m-n} \int_{-\infty}^\infty p_n(\xi)p_m(\xi) v(\xi) |g(\xi)|^2 \,\D \xi
$$
	This makes it clear that $\varphi$ is orthonormal with respect to the Sovolev inner product if and only if $P$ is orthonormal with respect to the measure $v(\xi)|g(\xi)|^2\mathrm{d}\xi$.
\end{proof}

\begin{remark}\label{rem:mollifier}
	There are infinitely many choices of $g$ which  satisfy \eqref{eqn:wtog}, namely
	\begin{displaymath}
		g(\xi) = \sqrt{\frac{w(\xi)}{v(\xi)}} \ee^{\ii \vartheta(\xi)},
	\end{displaymath}
for any measurable real-valued function $\vartheta$. Our canonical choice is $\vartheta \equiv 0$, although we know of no good reason, except for simplicity, why this might be superior to other choices.
\end{remark}

It is important to answer what space the resulting orthonormal system is dense in: ideally this is the inner product space
\begin{equation}\label{eqn:Hv}
	\CC{H}_v(\BB{R}) := \left\{ \psi \in \CC{L}_2(\BB{R}) : \langle \psi , \psi \rangle_v < \infty \right\},
\end{equation}
endowed with the inner product $\langle \cdot, \cdot \rangle_v$, but this need not be the case.

\begin{theorem}[Orthogonal bases of Paley--Wiener spaces]\label{thm:PW}
		Let $\Phi = \{\varphi_n\}_{n\in\bb{Z}_+}$ satisfy the requirements of Theorem \ref{thm:orthogonalPhi} with weight function $w(\xi)$ such that polynomials are dense in $\CC{L}_2(\BB{R};w(\xi)\mathrm{d}\xi)$. Then $\Phi$ forms a basis for the closure (in $\CC{H}_v(\BB{R})$) of the Paley--Wiener space $\mathcal{PW}_\Omega(\BB{R})$, where $\Omega$ is the support of $w$.
	\end{theorem}
	
	A proof of Theorem \ref{thm:PW}  can be obtained by modifying Theorem 9 from \cite{iserles19oss}. The key corollary is that for a basis $\Phi$ satisfying the requirements of Theorem \ref{thm:orthogonalPhi} to be complete in $\CC{L}_2(\BB{R})$, it is necessary that the polynomial basis $P$ is orthogonal with respect to a measure which is supported on the whole real line.

\setcounter{equation}{0}
\setcounter{figure}{0}
\section{Sobolev cascades}\label{sec:cascades}

In this section we derive two methods for producing orthonormal systems in the Sobolev space $\CC{H}^s(\BB{R})$ where $s = 0,1,2,\ldots$.

\subsection{Cascades of first and second kind}

For a weight function $w$ and $s \in \BB{Z}_+$ we can define the following two sequences of bases:

\begin{equation}
  \label{eq:first_kind}
	\varphi_n^{\langle s \rangle} (x) = \frac{\ii^n}{\sqrt{2\pi}} \int_{-\infty}^\infty \ee^{\ii x \xi} \, p_n(\xi) \, \sqrt{\frac{w(\xi)}{\sum_{k=0}^s \xi^{2k}}} \, \mathrm{d} \xi,
\end{equation}
where $P = \{p_n\}_{n\in\bb{Z}_+}$ are orthonormal polynomials with respect to $w(\xi)$, and
\begin{equation}
  \label{eq:second_kind}
	\varphi_n^{[ s]} (x) = \frac{\ii^n}{\sqrt{2\pi}} \int_{-\infty}^\infty \ee^{\ii x \xi} \, p^{[s]}_n(\xi) \,  \sqrt{w(\xi)} \, \mathrm{d} \xi,
	\end{equation}
where $P^{[s]} = \{p_n^{[s]}\}_{n\in\bb{Z}_+}$ are orthonormal polynomials with respect to the weight
\begin{displaymath}
	w^{[s]}(\xi) = \left( \sum_{k=0}^s \xi^{2k}\right) \!w(\xi)=\frac{1-\xi^{2(s+1)}}{1-\xi^2}w(\xi).
\end{displaymath}

By the theory described in Section \ref{sec:theory}, both systems $\Phi^{\langle s \rangle} = \{\varphi_n^{\langle s \rangle}\}_{n\in\bb{Z}_+}$ and $\Phi^{[s]}= \{\varphi_n^{[s]}\}_{n\in\bb{Z}_+}$ have skew-Hermitian tridiagonal differentiation matrices and both are orthonormal systems with respect to the standard $\CC{H}^s(\BB{R})$ Sobolev inner product described in the introduction. Furthermore, all of these systems are bases for (closure of) the Paley--Wiener space $\mathcal{PW}_\Omega(\BB{R})$, where $\Omega$ is the support of $w$.

We call the sequence $\Phi^{\langle 0 \rangle}, \Phi^{\langle 1 \rangle}, \Phi^{\langle 2 \rangle}\ldots$ a {\em Sobolev cascade of the first kind\/} for the weight function $w$, and  $\Phi^{[0]}, \Phi^{[1]}, \Phi^{[2]}\ldots$ a {\em Sobolev cascade of the second kind\/} for the weight function $w$.  Note that $\varphi_n^{[0]} = \varphi_n^{\langle 0 \rangle}$.

While a cascade of the first kind is perhaps a more natural generalisation of $\CC{L}_2$-orthogonality, it is also more problematic. Typically the polynomials $p_n$ might be already known, however the explicit form of the integrals \R{eq:first_kind}, hence of the $\varphi_n^{\langle s\rangle}$s, is often unknown, even for $s=1$. The issue with cascades of the second kind is different: the polynomials $P^{[s]}$ are usually unknown for $s\in\BB{N}$ even for the most familiar measures like Legendre or Hermite. On the other hand, once we know $p_n^{[s]}$ and can compute $\varphi_n^{[0]}$ explicitly, the closed form of $\varphi_n^{[s]}$ is available for all $n,s\in\BB{Z}_+$ through  the integral \R{eq:second_kind}. Note that to compute \R{eq:DefVarphi} in a closed form we need to be able to integrate explicitly Fourier transforms of $p\sqrt{w}$ for polynomials $p$: exactly the same is required for the computation of \R{eq:second_kind}.

\subsection{Sobolev cascades of the second kind}

Orthogonal systems in a cascade of the second kind have a simple relationship.

\begin{theorem}\label{thm:cascade}
 Let $s \in \BB{Z}_+$. There exists an infinite, lower triangular matrix $C^{[s]}$ which has bandwidth $2s$, such that
 \begin{equation}\label{eqn:cascade}
 	\boldsymbol{\varphi^{[0]}} = C^{[s]} \boldsymbol{\varphi^{[s]}}.
 \end{equation}
	\end{theorem}
\begin{proof}
	Since $p_n^{[0]}$ and $p_n^{[s]}$ are polynomials of degree $n$ (for every $n$), there exists an upper triangular \emph{connection coefficient matrix} $C^{[s]}$ such that
	\begin{equation}\label{eqn:connection}
		p_n^{[0]} = \sum_{j=0}^n \tilde{C}^{[s]}_{n,j} p_j^{[s]}.
	\end{equation}
Since $P^{[s]}$ is an orthonormal basis with respect to the weight function $\left(\sum_{k=0}^s \xi^{2k} \right) w(\xi)$, we have the formula
\begin{equation}
   \label{eq:ConCoeffs}
	\tilde{C}^{[s]}_{n,j} = \int_{-\infty}^\infty p_n^{[0]}(\xi) p_j^{[s]}(\xi)  \left(\sum_{k=0}^s \xi^{2k} \right) \! w(\xi) \,\mathrm{d}\xi.
\end{equation}
Since $p_j^{[s]}(\xi)  \left(\sum_{k=0}^s \xi^k \right)$ is a polynomial of degree at most $j+2s$, and $P^{[0]}$ is orthonormal with respect to $w$, we have that $C^{[s]}_{n,j} = 0$ if $j \leq n-2s-1$, which proves the desired bandwidth of the matrix. The proof is completed by multiplying equation \eqref{eqn:connection} by $\sqrt{w(\xi)}$ and taking the inverse Fourier transform:
\begin{equation}\label{eqn:phiconnection}
	\varphi_n^{[0]} = \sum_{j=0}^n C^{[s]}_{n,j} \varphi_j^{[s]}, \text{ where } C^{[s]}_{n,j} = \ii^{n-j} \tilde{C}^{[s]}_{n,j}.
\end{equation}
	\end{proof}
	
Note further that if the weight function $w$ is symmetric then all the polynomials $p_n^{[s]}$ maintain the parity of $n$ and it follows easily that $C_{n,j}^{[s]}=0$ for $n+j$ odd.

Theorem \ref{thm:cascade} has two consequences. Firstly, if one can calculate $\{\varphi_0^{[0]}, \varphi_1^{[0]}, \ldots, \varphi_N^{[0]}\}$, then it is possible to calculate $\{\varphi_0^{[s]}, \varphi_1^{[s]}, \ldots, \varphi_N^{[s]}\}$ in $\mathcal{O}(N)$ operations by applying forward substitution to the banded lower triangular system with matrix $C^{[s]}$.

Secondly,  given $N+1$ expansion coefficients in the basis $\Phi^{[0]}$, we can compute the equivalent expansion coefficients in the basis $\Phi^{[s]}$ in $\mathcal{O}(N)$ operations. Specifically, if
\begin{displaymath}
	\sum_{n=0}^N a_n^{[0]} \varphi_n^{[0]}(x) = 	\sum_{n=0}^N a_n^{[s]} \varphi_n^{[s]}(x),
\end{displaymath}
then
\begin{equation}
  \label{eq:BackSubst}
	{C^{[s]}}^{\top}\! \boldsymbol{a^{[s]}} = 	\boldsymbol{a^{[0]}},
\end{equation}
which can be solved in $\mathcal{O}(N)$ operations by back substitution.

In general, it appears that the nonzero entries of $C_{n,j}^{[s]}$ obey no recognisable numerical relations: for example, the $6\times6$ principal minor of $C^{[1]}$ for the Hermite weight is
\begin{displaymath}
  \left[
  \begin{array}{cccccc}
    \sqrt{\frac32} & 0 & 0 & 0 & 0 & 0\\
    0 & \sqrt{\frac52} & 0 & 0 & 0 & 0\\
    \sqrt{\frac13} & 0 & \sqrt{\frac{19}{6}} & 0 & 0 & 0\\
    0 & \sqrt{\frac35} & 0 & \sqrt{\frac{39}{10}} & 0 & 0\\
    0 & 0 & \sqrt{\frac{18}{19}} & 0 & \sqrt{\frac{173}{38}} & 0\\
    0 & 0 & 0 & \sqrt{\frac{819}{407}} & 0 & \sqrt{\frac{407}{78}}
  \end{array}
  \right]\!.
\end{displaymath}
It is difficult to discern a pattern: numerical experiments for large values of $n$ indicate that both $C_{n,n}^{[1]}$ and $C_{n+2,n}^{[1]}$ grow like $\O{\sqrt{n}}$.

All this does not rule out computing the  $C^{[s]}_{n,j}$s numerically. Modifying a weight by a quadratic factor and computing the connection coefficients is discussed in \cite{gautschi2004orthogonal} and \cite{golub2009matrices}. We won't go into the details of this in this paper.

\setcounter{equation}{0}
\setcounter{figure}{0}
\section{Hermite-type systems}

\subsection{The Hermite--Sobolev cascade of the first kind}

A natural starting point is the Hermite weight $w(\xi)=\ee^{-\xi^2}$, $\xi\in\BB{R}$, and $s = 1$. The mollifier, by the definitions in Section \ref{sec:cascades}, is $g(\xi)=\ee^{-\xi^2}/(1+\xi^2)^{1/2}$, so
\begin{equation}
  \label{eq:SHermite}
  \varphi_n(x)=\frac{\ii^n}{\sqrt{2\pi}} \int_{-\infty}^\infty \tilde{\CC{H}}_n(\xi) \sqrt{\frac{\ee^{-\xi^2}}{1+\xi^2}} \ee^{\ii x\xi}\D\xi,\qquad n\in\BB{Z}_+,
\end{equation}
where
\begin{displaymath}
  \tilde{\CC{H}}_n(\xi)=\frac{1}{\sqrt{2^nn!\sqrt{\pi}}} \CC{H}_n(\xi),\qquad n\in\BB{Z}_+,
\end{displaymath}
are the orthonormalised Hermite polynomials. Unfortunately, the integrals \R{eq:SHermite} are not known in an explicit form, not even $\varphi_0$.

\begin{figure}[htb]
\begin{center}
  \includegraphics[width=120pt]{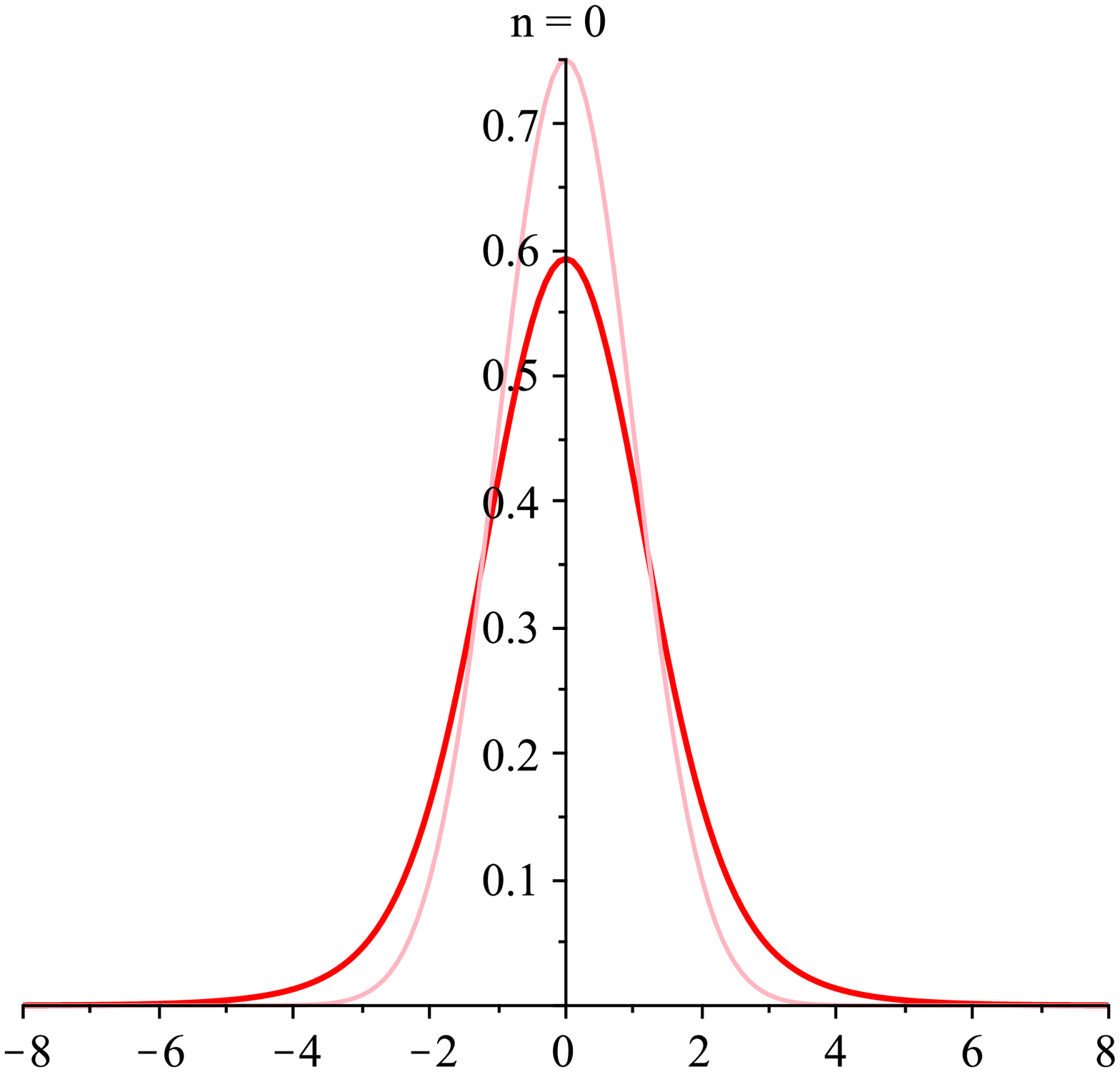}\hspace*{5pt}\includegraphics[width=120pt]{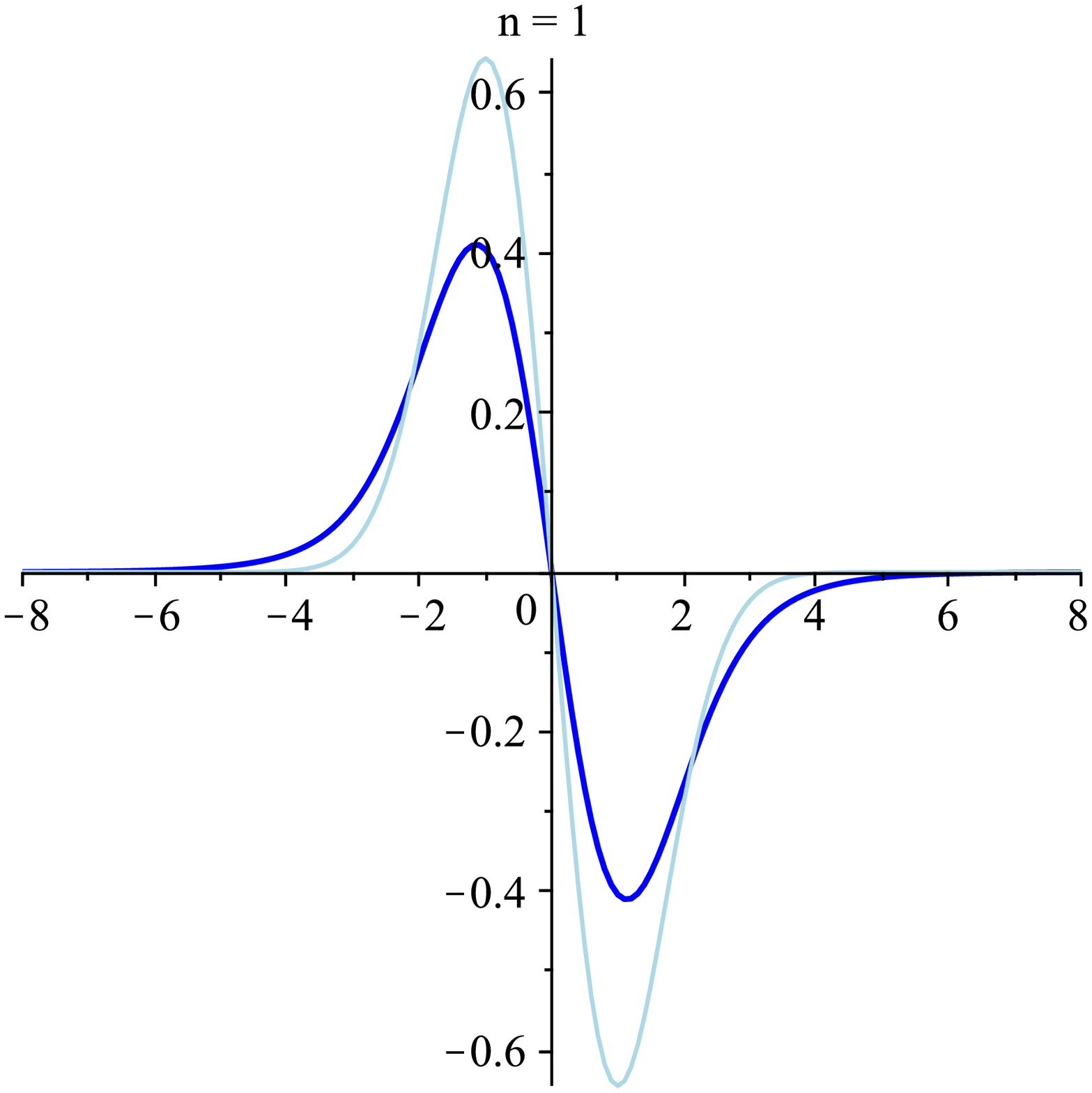}\hspace*{5pt}\includegraphics[width=120pt]{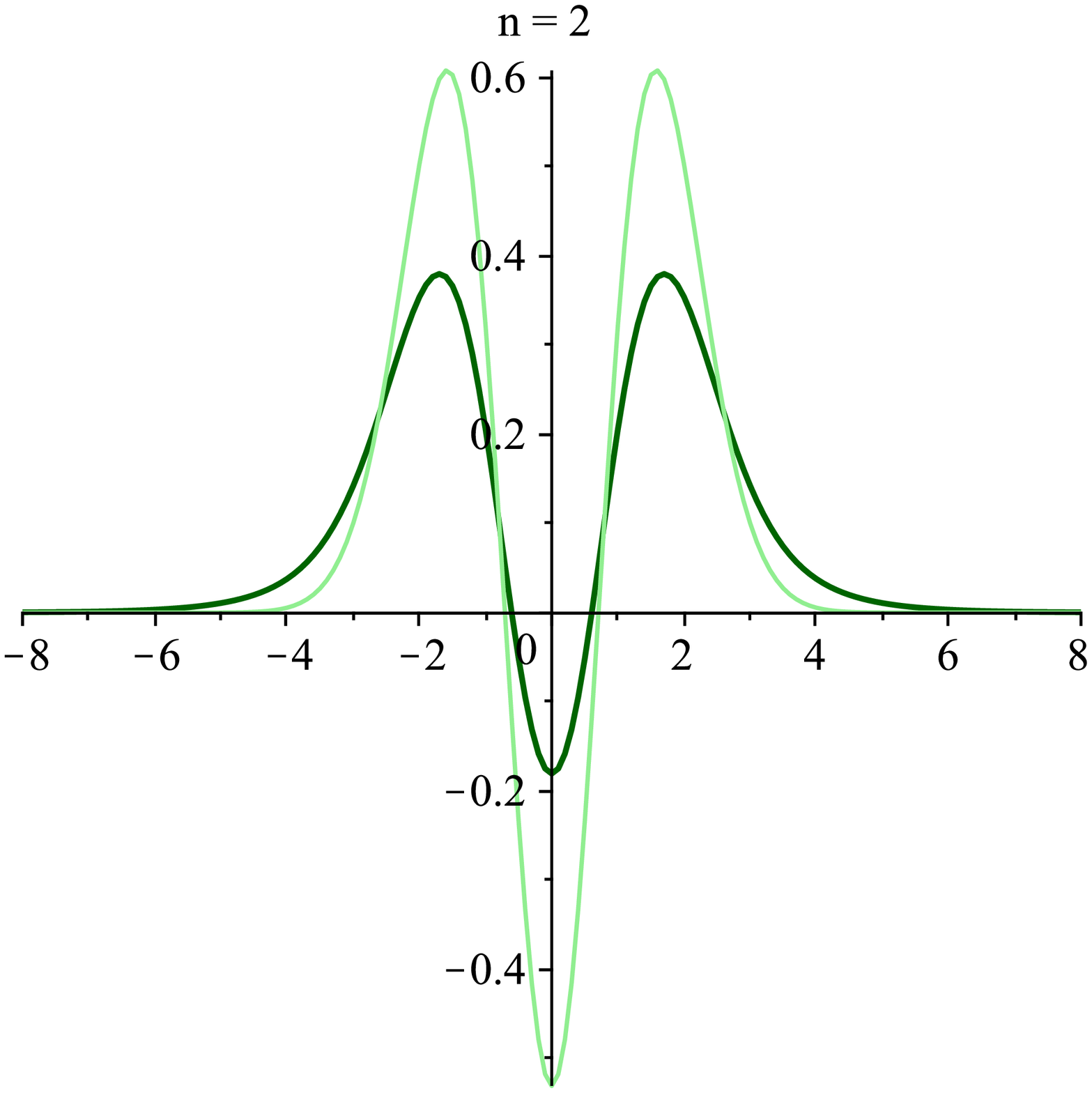}
  \includegraphics[width=120pt]{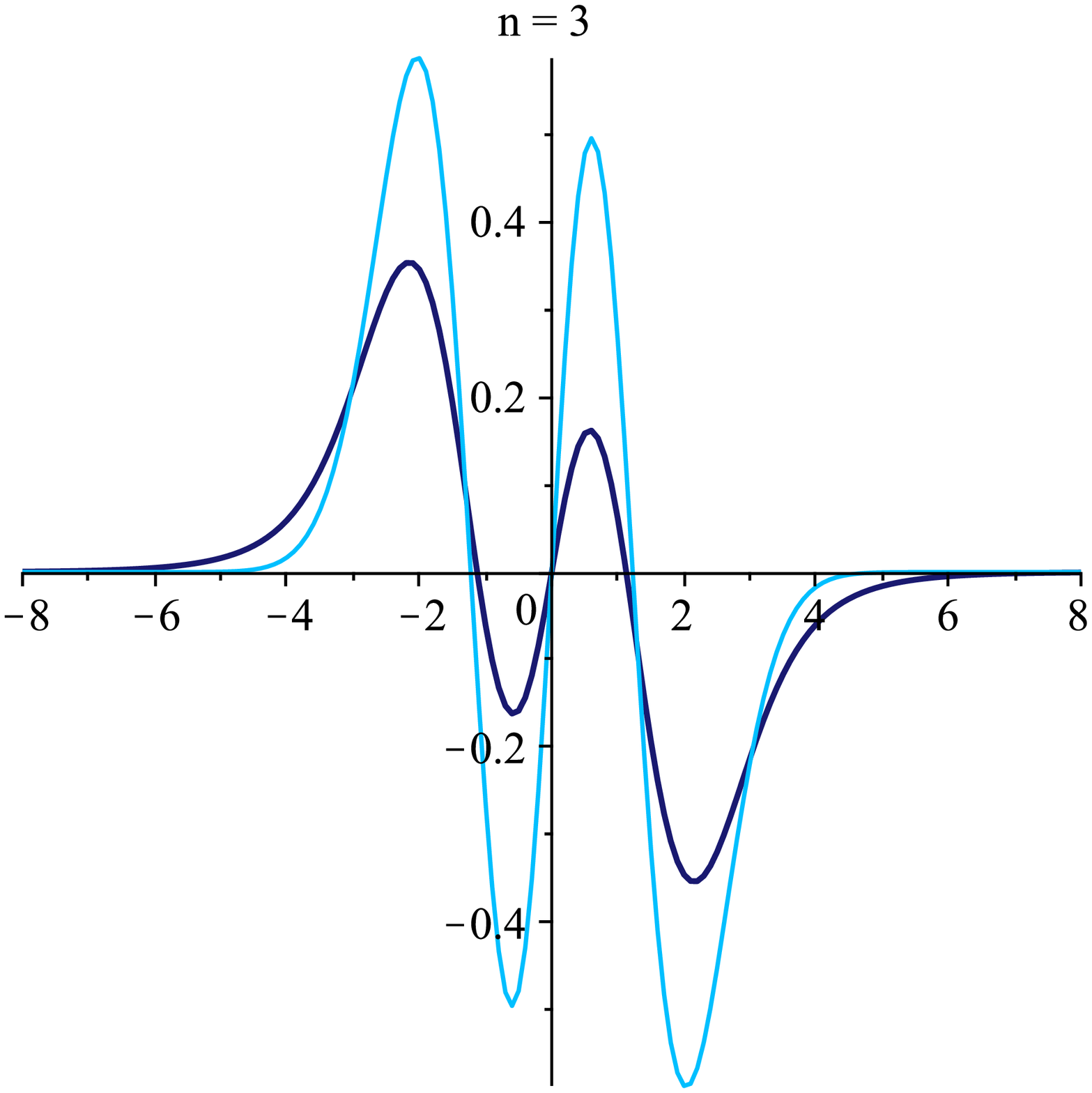}\hspace*{5pt}\includegraphics[width=120pt]{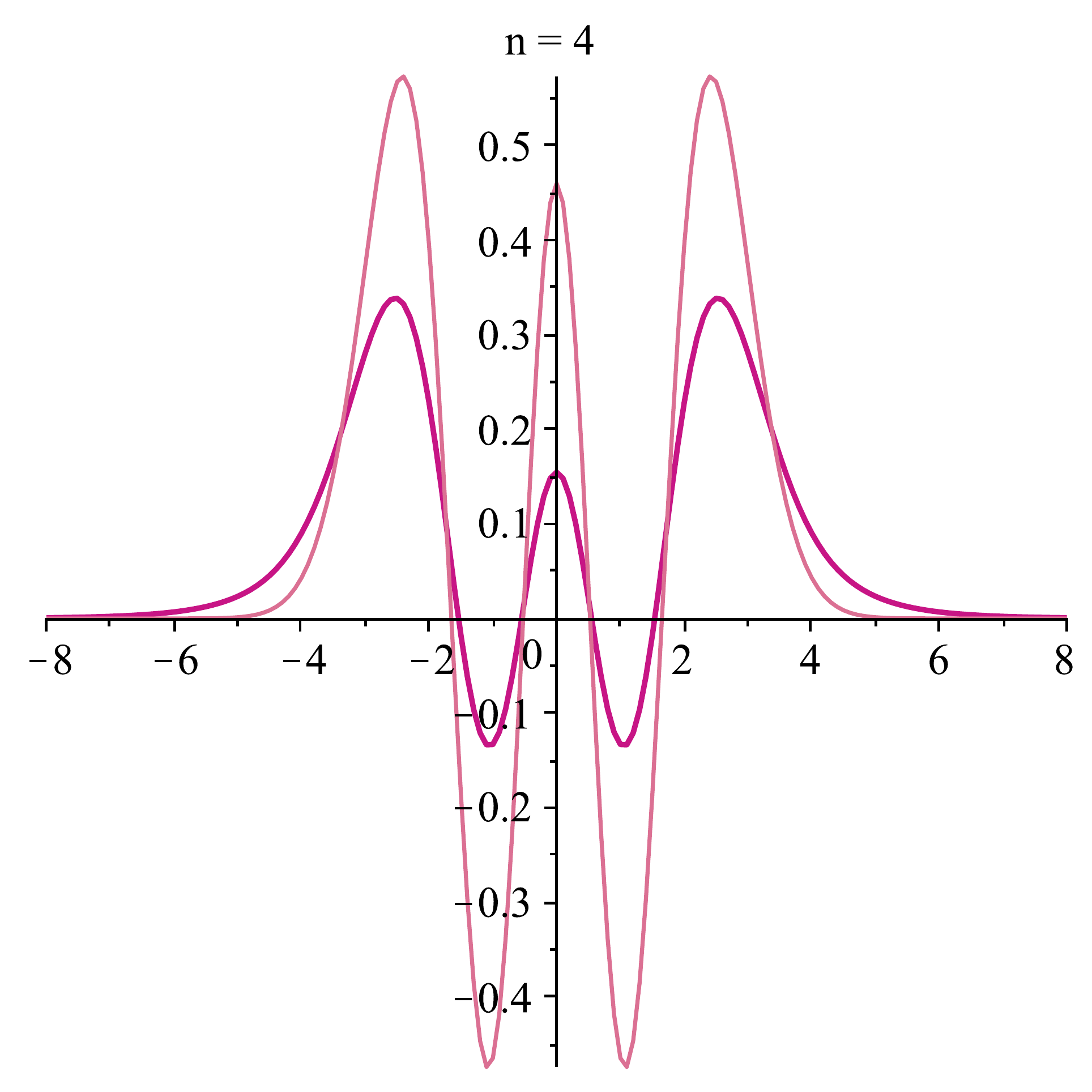}\hspace*{5pt}\includegraphics[width=120pt]{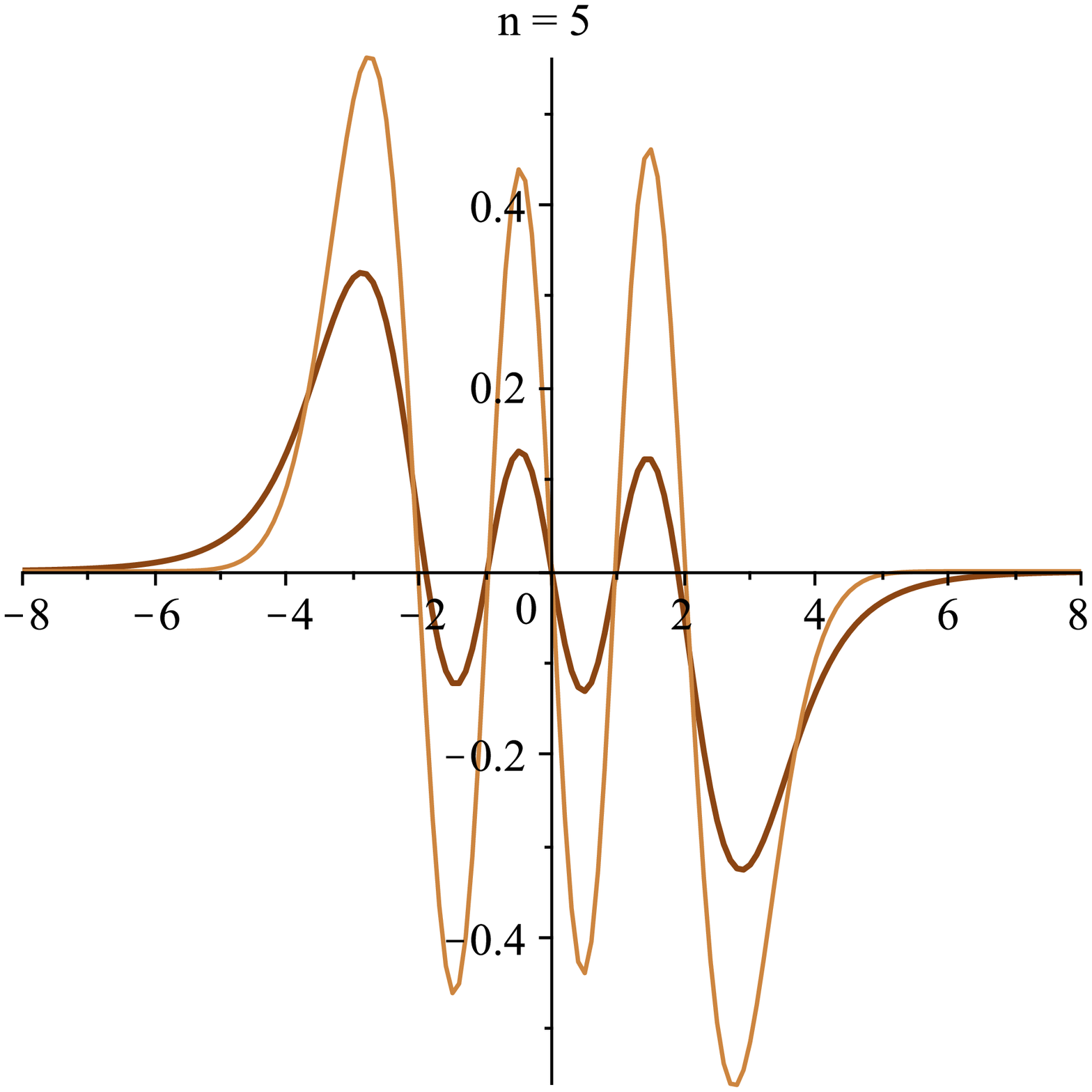}
\caption{The first six functions $\varphi_n^{\langle1\rangle}$ defined by \R{eq:SHermite}, with corresponding Hermite functions $\varphi_n^{\langle0\rangle}(x)=(-1)^n\ee^{-x^2/2}\tilde{\CC{H}}_n(x)$ in darker shade.}
\label{fig:3.1}
\end{center}
\end{figure}

In Fig.~\ref{fig:3.1} we display the functions $\varphi_n$, $n=0,\ldots,5$, computed by brute-force numerical quadrature. In the background, in fainter colour, we display the familiar Hermite functions which follow from \R{eq:DefVarphi} and are orthonormal in $\CC{L}_2(\BB{R})$ (while, by Theorem~2, the $\varphi_n$s are orthonormal in $\CC{H}^1(\BB{R})$). 

\subsection{The Hermite--Sobolev cascade of the second kind}

While the polynomials $p_n^{[s]}$ from Subsection~3.1 are unknown for $s\in\BB{N}$, it is possible to generate them, as explained in Section~3 or directly from the moments: in the simplest nontrivial case, $s=1$, the moments are
\begin{displaymath}
  \mu_n=\int_{-\infty}^\infty \xi^n (1+\xi^2)\ee^{-\xi^2}\D \xi=
  \begin{case}
     \displaystyle \frac{\sqrt{\pi}(2n)!(2n+3)}{2^{2n+1}n!}, & n\mbox{\ even,}\\[4pt]
     0, & n\mbox{\ odd,}
  \end{case}
\end{displaymath}
and the first few $p_n^{[1]}$s are
\begin{Eqnarray*}
  p_0^{[1]}(\xi)&\equiv&\frac{\sqrt{6}}{3\pi^{1/4}},\\
  p_1^{[1]}(\xi)&=&\frac{2\sqrt{5}}{5\pi^{1/4}} \xi,\\
  p_2^{[1]}(\xi)&=&\frac{2\sqrt{57}}{19\pi^{1/4}} \left(\xi^2-\frac56\right)\!,\\
  p_3^{[1]}(\xi)&=&\frac{2\sqrt{130}}{39\pi^{1/4}}\left(\xi^3-\frac{21}{10}\xi\right)\!,\\
  p_4^{[1]}(\xi)&=&\frac{2\sqrt{9861}}{519\pi^{1/4}} \left(\xi^4-\frac{75}{19}\xi^2+\frac{117}{76}\right)\!,\\
  p_5^{[1]}(\xi)&=&\frac{2\sqrt{52910}}{2035\pi^{1/4}} \left(\xi^5-\frac{245}{39}\xi^3+\frac{335}{52}\xi\right)
\end{Eqnarray*}
and so on. Likewise, it is possible to compute recurrence coefficients,
\begin{displaymath}
  b_0=\sqrt{\frac{5}{6}},\; b_1=\sqrt{\frac{19}{15}},\; b_2=\sqrt{\frac{315}{190}},\; b_3=\sqrt{\frac{1730}{741}},\; b_4=\sqrt{\frac{38665}{13494}},\; b_5=\sqrt{\frac{236925}{70411}}
\end{displaymath}
etc.\ but difficult to discern any pattern except for the obvious, $b_n= \O{n^{1/2}} $, $n\gg1$, a consequence of the proof of the Freud conjecture in \cite{lubinsky88pfc}. Likewise, we can compute $p_n^{[s]}$ for $s\geq2$: Fig.~\ref{fig:3.1half} displays $p_n^{[s]}$ for $n=2,3,4,5$ and $s=0,1,2,3,4$.

\begin{figure}[htb]
\begin{center}
   \includegraphics[width=180pt]{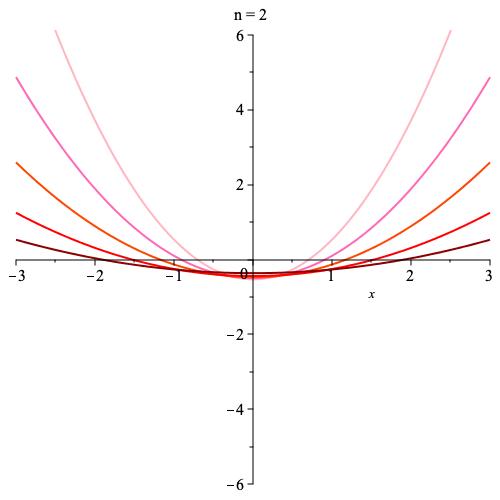}\hspace*{15pt}\includegraphics[width=180pt]{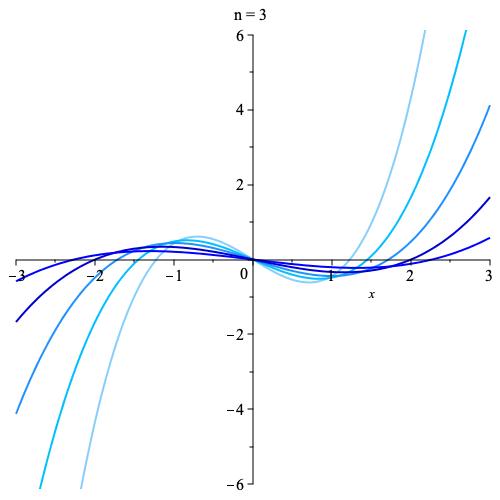}
   
   \includegraphics[width=180pt]{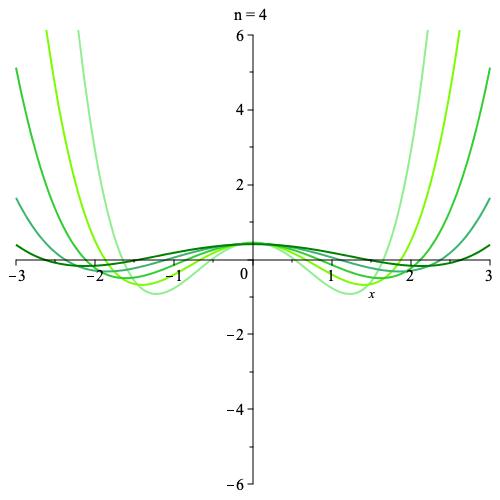}\hspace*{15pt}\includegraphics[width=180pt]{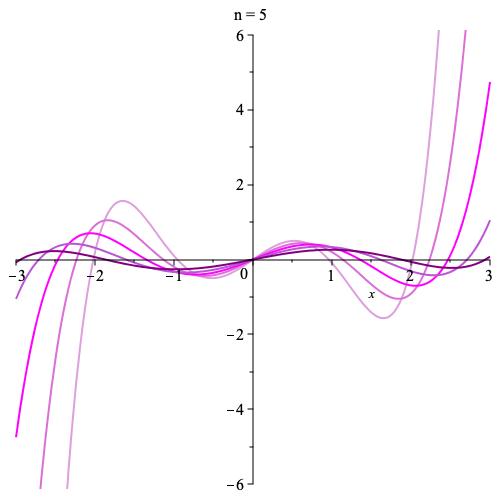}
\caption{The polynomials $p_n^{[s]}$ for $n=2,3,4,5$ and $s=0,1,2,4$ (darker hue corresponds to larger $s$).}
\label{fig:3.1half}
\end{center}
\end{figure}

 Computing the $\varphi_n^{[1]}$s in line with \R{eq:def_varphi} is straightforward:
\begin{Eqnarray*}
  \varphi_0^{[1]}(x)&=&\sqrt{\frac{2}{3}} \pi^{-1/4} \ee^{-x^2/2},\\
  \varphi_1^{[1]}(x)&=&-\sqrt{\frac{4}{5}} \pi^{-1/4} x\ee^{-x^2/2},\\
  \varphi_2^{[1]}(x)&=&\frac{1}{\sqrt{57}} \pi^{-1/4} (6x^2-1)\ee^{-x^2/2},\\
  \varphi_3^{[1]}(x)&=&-\sqrt{\frac{2}{585}}\pi^{-1/4} (10x^3-9x)\ee^{-x^2/2},\\
  \varphi_4^{[1]}(x)&=&\frac{1}{\sqrt{39444}} \pi^{-1/4} (76x^4-156x^2+45)\ee^{-x^2/2},\\
  \varphi_5^{[1]}(x)&=&-\frac{1}{\sqrt{476190}} \pi^{-1/4} (156x^5-580x^3+405x)\ee^{-x^2/2}
\end{Eqnarray*}
and so on. 

\begin{figure}[htb]
\begin{center}
  \includegraphics[width=120pt]{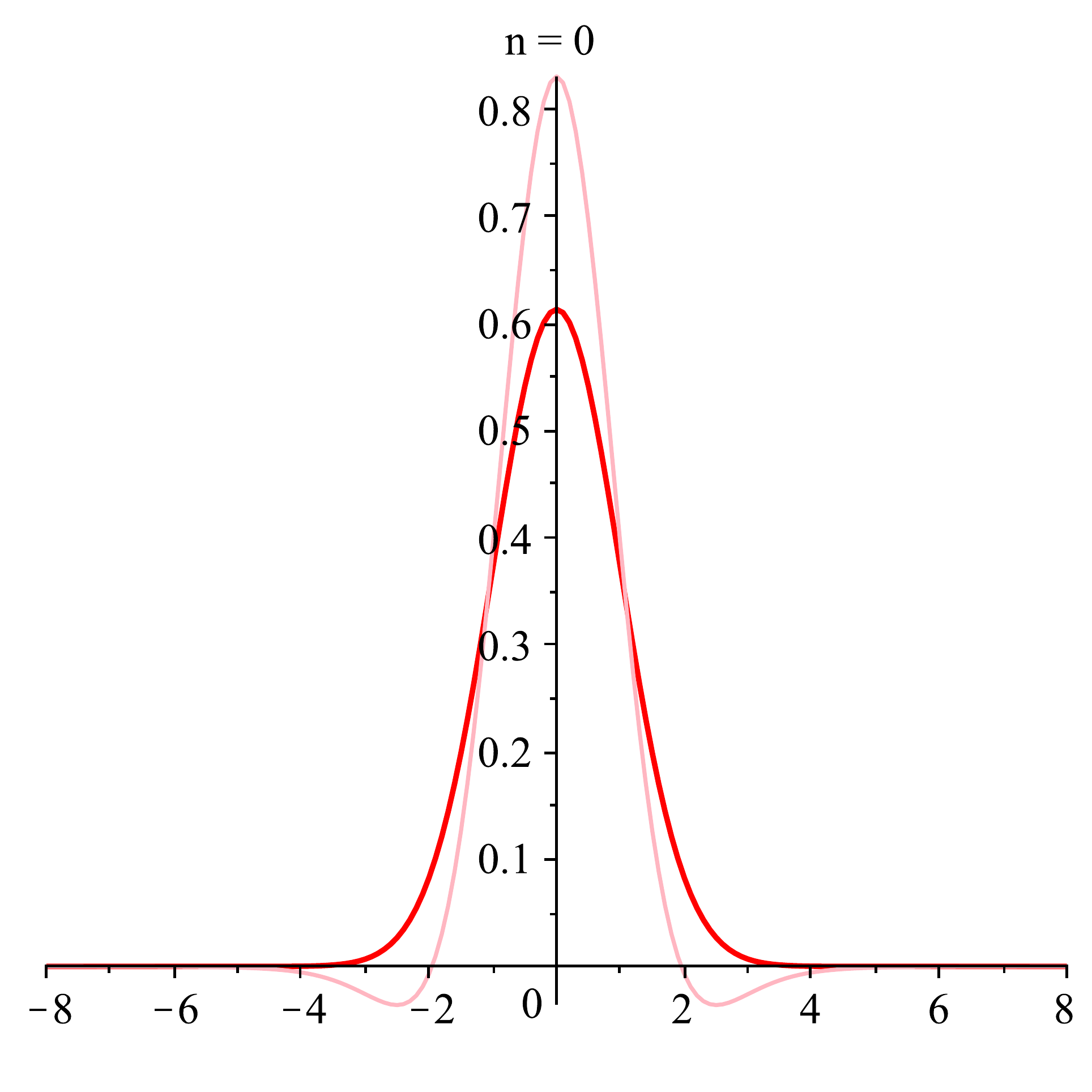}\hspace*{5pt}\includegraphics[width=120pt]{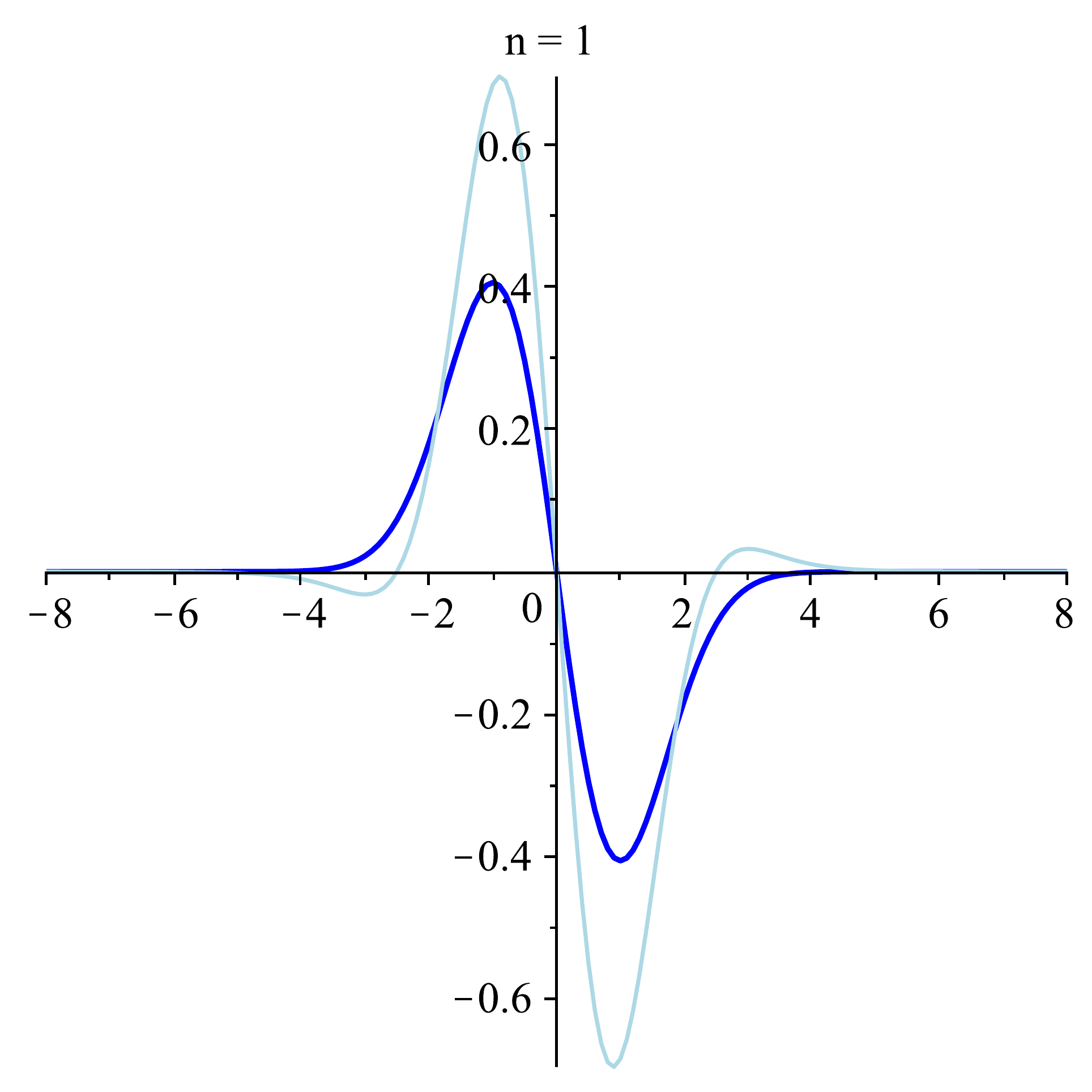}\hspace*{5pt}\includegraphics[width=120pt]{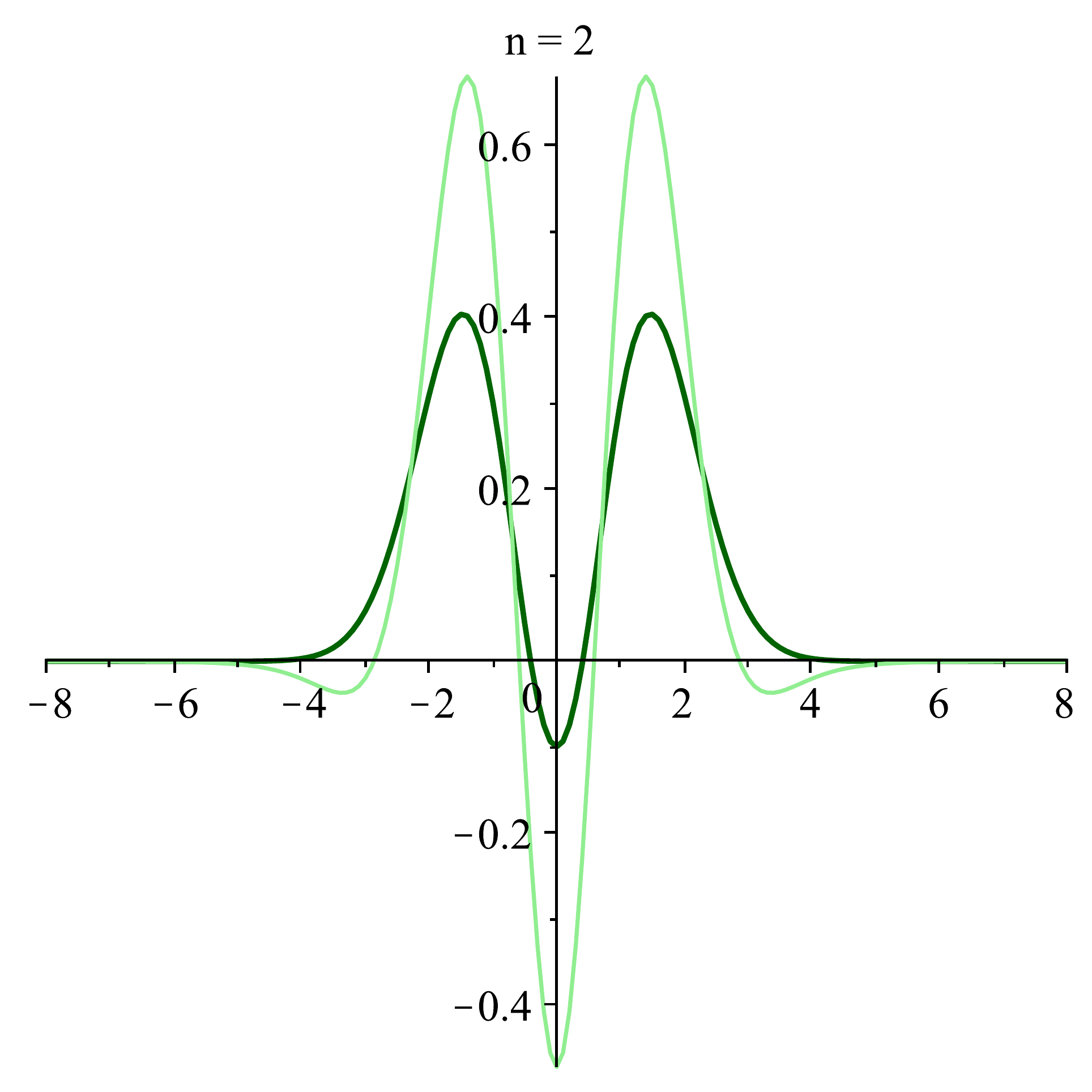}
  \includegraphics[width=120pt]{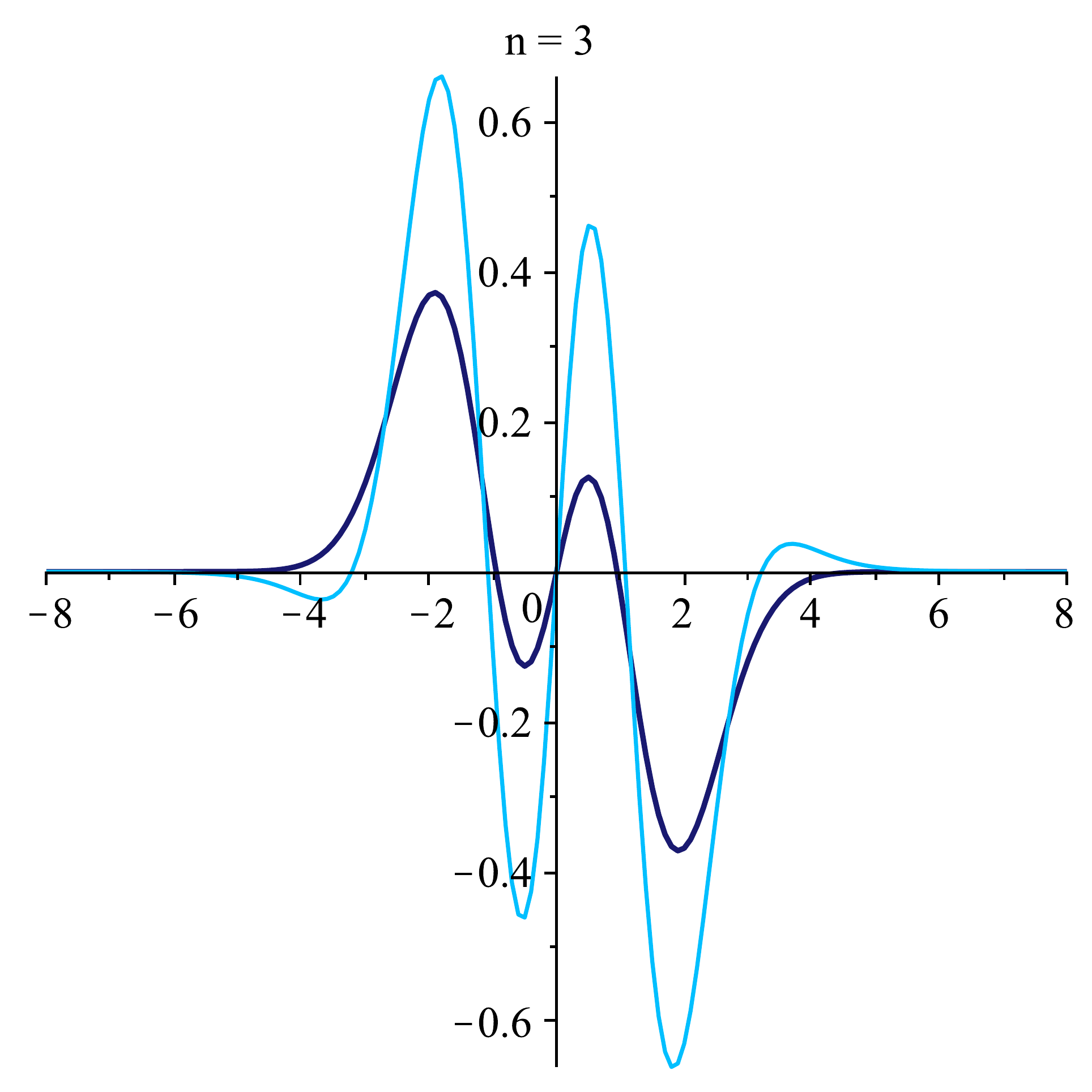}\hspace*{5pt}\includegraphics[width=120pt]{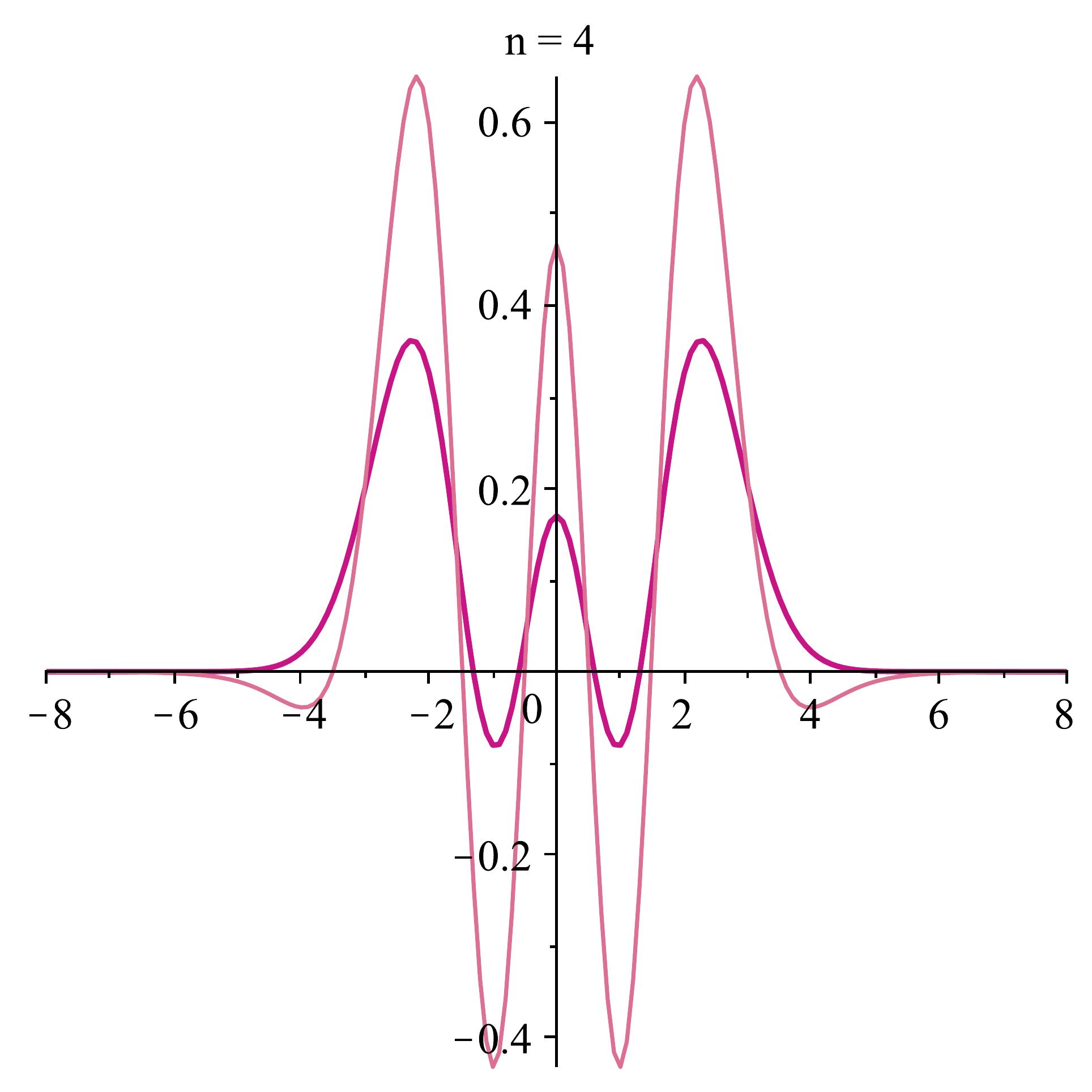}\hspace*{5pt}\includegraphics[width=120pt]{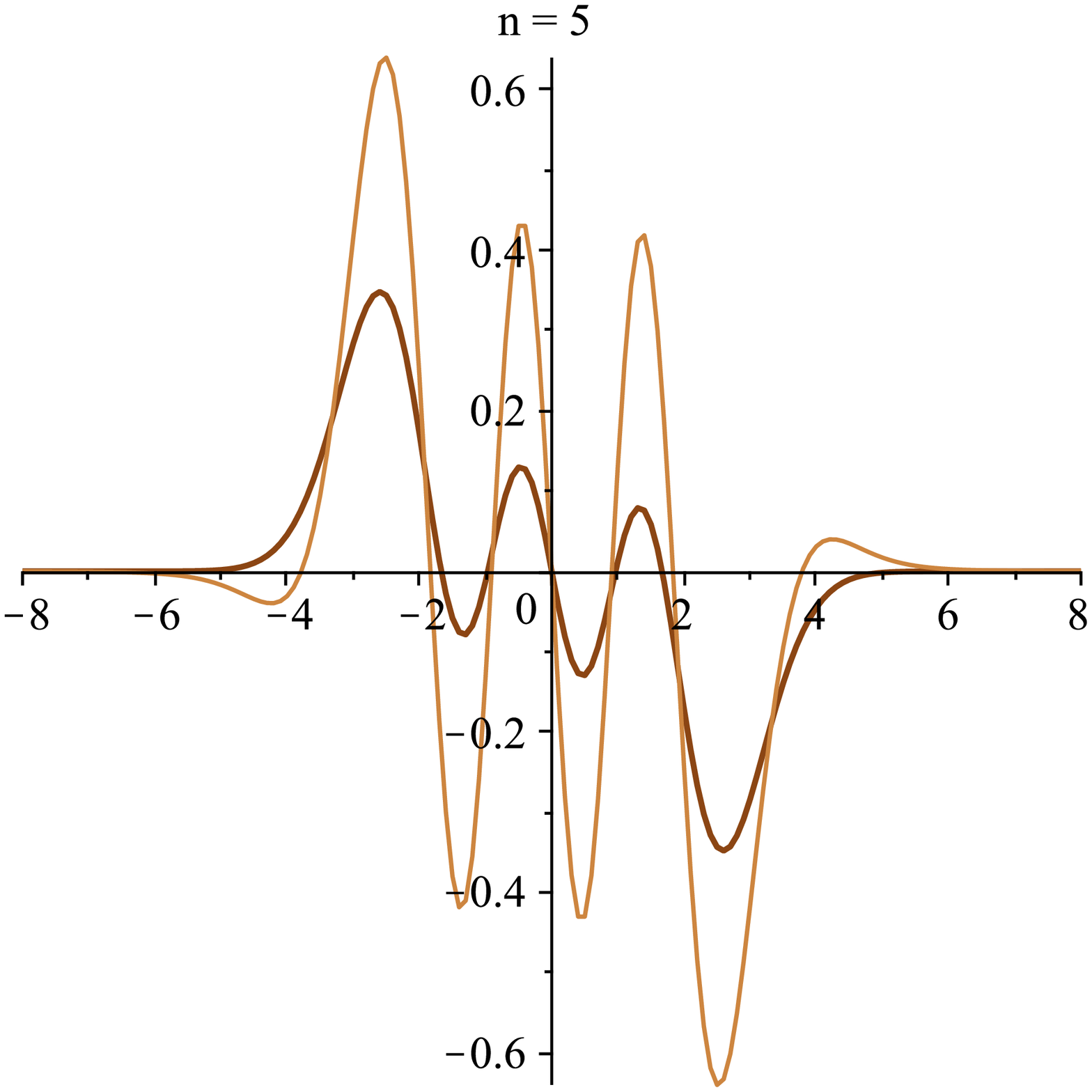}
\caption{The first six functions $\varphi_n^{[1]}$ with corresponding  functions $\varphi_n^{[0]}$, which are orthogonal in $\CC{L}_2(\BB{R})$, in darker shade.}
\label{fig:3.2}
\end{center}
\end{figure}

Fig.~\ref{fig:3.2} displays the above functions $\varphi_n^{[1]}$ and, in fainter colour, the functions $\varphi_n^{\langle0\rangle}$ based on the same weight $w(\xi)=(1+\xi^2)\ee^{-\xi^2}$ and defined by \R{eq:DefVarphi}. 

\begin{lemma}
 For every $n\in\BB{Z}_+$ we have $\varphi_n(x)=\lambda_n(x)\ee^{-x^2/2}$, where $\lambda_n$ is an $n$th-degree polynomial.
\end{lemma}

\begin{proof}
 It is enough to prove that
 \begin{displaymath}
   \sigma_n(x)=\frac{\ii^n}{\sqrt{2\pi}} \int_{-\infty}^\infty \xi^n \ee^{-\xi^2/2+\ii x\xi}\D\xi,\qquad n\in\BB{Z}_+,
 \end{displaymath}
  is of the asserted form, i.e.\ an $n$th degree polynomial times $\ee^{-x^2/2}$. This follows readily by induction on $n$ from $\sigma_0(x)=\ee^{-x^2/2}$ and $\sigma_n'=\sigma_{n+1}$ because, letting $\sigma_n(x)=\alpha_n(x)\ee^{-x^2/2}$, we obtain $\alpha_{n+1}(x)=\alpha_n'(x)-x\alpha_n(x)$.
\end{proof}

Alternatively, substituting into \R{eq:3term}, it is easy to see that
\begin{displaymath}
  \lambda_n'(x)=-b_{n-1}\lambda_{n-1}(x)+x\lambda_n(x)+b_n\lambda_{n+1}(x),\qquad n\in\BB{Z}_+.
\end{displaymath}
The proof that $\lambda_n$ is an $n$th degree polynomial follows at once by induction on this differential recurrence, since $b_n>0$, $n\in\BB{N}$.

The bad news is that the $\lambda_n$s are not known and, as is trivial to verify, they do not obey a three-term recurrence relation (hence, by the Favard theorem, cannot be orthogonal with respect to any Borel measure). Having said so, the result is highly interesting. It has been proved in \cite{iserles19oss} that there exists a unique $\CC{L}_2$-orthonormal system on the real line which obeys \R{eq:TTR} and where each function is a polynomial multiple of the same $\CC{L}_2$ function, specifically Hermite functions (or $\varphi_n^{[0]}=\varphi_n^{\langle0\rangle}$ in present notation). The functions $\{\varphi_n^{[1]}\}_{n\in\bb{Z}_+}$, though, are $\CC{H}^1$-orthonormal, they obey \R{eq:TTR} and $\varphi_n^{[1]}(x)=\ee^{-x^2/2}\lambda_n(x)$. 

\begin{lemma}
 The only $\CC{H}_v(\BB{R})$-orthonormal systems (see equation \eqref{eqn:Hv}) with a tridiagonal differentiation matrix which are of the form $\varphi_n(x)=G(x)\lambda_n(x)$, $n\in\BB{Z}_+$, for some function $G\in\CC{L}_2(\BB{R})$, $G>0$ (and $G(0) = 1$ without loss of generality), where each $\lambda_n$ is a polynomial of degree $n$, correspond to
 \begin{equation}\label{eqn:G}
 	G(x) = \exp\left(-\gamma x^2 + \delta x\right)
 \end{equation}
for some constants $\gamma > 0$ and $\delta \in \BB{R}$. The corresponding weight of orthonormality for $P$ in Theorem \ref{thm:vanilla} is
\begin{equation}\label{eqn:Gw}
w(\xi) \propto v(\xi) \ee^{-\xi^2/(2\gamma)}.
\end{equation}
\end{lemma}

\begin{proof}
 We substitute $\varphi_n(x)=G(x)\lambda_n(x)$ into \R{eq:3term}, bearing in mind that $G>0$, to obtain,
 \begin{displaymath}
    \lambda_n'(x)=-b_{n-1}\lambda_{n-1}(x)+\left(\ii c_n - \frac{G'(x)}{G(x)}\right)\lambda_n(x)+b_n\lambda_{n+1}(x),\qquad n\in\BB{Z}_+.
 \end{displaymath}
 Since $\deg\lambda_m=m$ by assumption, we deduce, comparing degrees, that $G'/G$ is a linear polynomial, and hence that $G(x)$ is the exponential of a quadratic polynomial. We can set the constant term in this quadratic to zero since $G(0) = 1$ without loss of generality, so we obtain equation \eqref{eqn:G}.
 
 Inverting the representation in Theorem \ref{thm:vanilla}, we have
 \begin{displaymath}
    p_n(\xi)g(\xi)=\frac{(-\ii)^n}{\sqrt{2\pi}}\int_{-\infty}^\infty \varphi_n(x)\ee^{-\ii x\xi}\D x=\frac{(-\ii)^n}{\sqrt{2\pi}}\int_{-\infty}^\infty \lambda_n(x)\ee^{-\gamma x^2\xi + \delta x-\ii x\xi}\D x.
  \end{displaymath}
The case $n = 0$ tells us that
\begin{equation}
	g(\xi) \propto \exp(- (\xi - \ii \delta)^2 / (4\gamma)).
\end{equation}
 Theorem~2 tells us that for $\CC{H}_v(\BB{R})$ orthonormality we require
  \begin{equation}
     \label{eq:Omega}
     w(\xi) = v(\xi)|g(\xi)|^2,
  \end{equation}
which completes the proof of necessity of the forms of $G$ and $w$.

Now we prove sufficiency. Let $w(\xi) = C^2 v(\xi) \ee^{-\xi^2 /2\gamma}$ where $C$ ensures $w$ has unit integral, $g(\xi) = C \ee^{-(\xi-\ii\delta)^2/(4\gamma)}$, and define $\Phi$ as in Theorem \ref{thm:vanilla}. By Theorem \ref{thm:orthogonalPhi}, $\Phi$ is an $\CC{H}_v(\BB{R})$-orthonormal system, so all that remains to prove is that $\varphi_n(x) = G(x)\lambda_n(x)$ where $\lambda_n$ is a polynomial of degree $n$. It is sufficient to show that $\rho_n(x) = \int_{-\infty}^\infty \xi^n g(\xi) \ee^{\ii x \xi} \, \D \xi$ is $G(x)$ times a polynomial of degree $n$, which can be readily shown by induction starting from $\rho_0(x) \propto G(x)$ and leveraging $\rho_{n+1} (x) = -\ii \rho_n'(x) $.
\end{proof}

\subsection{An $\CC{H}^\infty(\BB{R})$ system based on the Hermite weight}

Let $\sigma\in(0,1)$, $w(\xi)=\ee^{-\xi^2}$ (i.e.\ the standard Hermite weight) and $v(\xi)=\ee^{\sigma \xi^2}$, $\xi\in\BB{R}$. Therefore, by Theorem~3, the functions $\varphi_n$, as defined by \R{eqn:phinformula}, are orthogonal with respect to the infinite Sobolev inner product
\begin{equation}
  \label{eq:H_infty}
  \langle f,g\rangle_v=\sum_{\ell=0}^\infty \frac{\sigma^\ell}{\ell!} \int_{-\infty}^\infty f^{(\ell)}(x)g^{(\ell)}(x)\D x.
\end{equation}
In this case $p_n$s are scaled Hermite polynomials and $\varphi_n$s can be computed explicitly. 

\begin{theorem}
 The Hermite weight $w(\xi)=\ee^{-\xi^2}$, $x\in\BB{R}$, generates the $\CC{H}^\infty(\BB{R})$ system
 \begin{equation}
   \label{eq:H_infty_phi}
   \varphi_n^{[\infty]}(x)=\frac{1}{\sqrt{1+\sigma}} \left(\frac{1-\sigma}{1+\sigma}\right)^{\!n/2} \varphi_n^{[0]}\!\left(\frac{x}{\sqrt{1-\sigma^2}}\right) \!\exp\!\left(\frac{\sigma x^2}{2(1-\sigma^2)}\right)\!,\qquad n\in\BB{Z}_+,
\end{equation}
where $\varphi_n^{[0]}$ is the standard $n$th Hermite function.\end{theorem}

\begin{proof}
Let
\begin{displaymath}
  \tilde{\varphi}_n(x)=\frac{\ii^n}{\sqrt{2\pi}} \int_{-\infty}^\infty \CC{H}_n(\xi)\ee^{-\frac12(1+\sigma)\xi^2+\ii x\xi}\D\xi,\qquad n\in\BB{Z}_+,
\end{displaymath}
whereby, orthonormalising Hermite polynomials,  \R{eqn:phinformula} yields $\varphi_n(x)=\tilde{\varphi}_n(x)/\sqrt{2^nn!\sqrt{\pi}}$. Using the standard generating function for Hermite polynomials,
\begin{Eqnarray*}
  \sum_{n=0}^\infty \frac{\tilde{\varphi}_n(x)}{n!}t^n&=&\frac{1}{\sqrt{2\pi}} \int_{-\infty}^\infty \left[\sum_{n=0}^\infty \frac{\CC{H}_n(\xi)}{n!} (\ii t)^n\right]\! \exp\!\left(-\frac{1+\sigma}{2}\xi^2+\ii x\xi\right)\!\D\xi\\
  &=&\frac{1}{\sqrt{2\pi}} \int_{-\infty}^\infty \exp\!\left(2\ii\xi t+t^2-\frac{1+\sigma}{2}\xi^2+\ii x\xi\right)\!\D\xi \\
  &=&\frac{1}{\sqrt{1+\sigma}} \exp\!\left(-\frac{x^2}{2(1+\sigma)}\right)\! \exp\!\left(-\frac{2xt}{1+\sigma} -\frac{1-\sigma}{1+\sigma}t^2\right)\\
  &=&\frac{1}{\sqrt{1+\sigma}} \exp\!\left(-\frac{x^2}{2(1+\sigma)}\right)\!\exp\!\left(-\frac{2x}{\sqrt{1-\sigma^2}} \left(\sqrt{\frac{1\!-\!\sigma}{1\!+\!\sigma}} t\right) -\left(\sqrt{\frac{1\!-\!\sigma}{1\!+\!\sigma}} t\right)^{\!2}\right)
\end{Eqnarray*}
and, using the same generating function in the opposite direction, 
\begin{Eqnarray*}
  \sum_{n=0}^\infty \frac{\tilde{\varphi}_n(x)}{n!}t^n&=&\frac{1}{\sqrt{1+\sigma}} \exp\!\left(-\frac{x^2}{2(1+\sigma)}\right)\!\sum_{n=0}^\infty \frac{1}{n!} \CC{H}_n\!\left(-\frac{x}{\sqrt{1-\sigma^2}}\right)\!\left(\sqrt{\frac{1-\sigma}{1+\sigma}}t\right)^{\!n}.
\end{Eqnarray*}
Therefore
\begin{displaymath}
  \tilde{\varphi}_n(x)=\frac{(-1)^n}{\sqrt{1+\sigma}} \left(\frac{1-\sigma}{1+\sigma}\right)^{\!n/2} \CC{H}_n\!\left(\frac{x}{\sqrt{1-\sigma^2}}\right)\!\exp\!\left(-\frac{x^2}{2(1+\sigma)}\right)\!.
\end{displaymath}
Normalising,
\begin{Eqnarray*}
  \varphi_n^{[\infty]}(x)&=&\frac{(-1)^n}{\sqrt{(1+\sigma)2^nn!\sqrt{\pi}}} \left(\frac{1-\sigma}{1+\sigma}\right)^{\!n/2} \CC{H}_n\!\left(\frac{x}{\sqrt{1-\sigma^2}}\right)\!\exp\!\left(-\frac{x^2}{2(1+\sigma)}\right)\!,\\
  &=&\frac{1}{\sqrt{1+\sigma}} \left(\frac{1-\sigma}{1+\sigma}\right)^{\!n/2} \varphi_n^{[0]}\!\left(\frac{x}{\sqrt{1-\sigma^2}}\right) \!\exp\!\left(\frac{\sigma x^2}{2(1-\sigma^2)}\right)\!,\qquad n\in\BB{Z}_+,
\end{Eqnarray*}
and \R{eq:H_infty_phi} is true.
\end{proof}

\begin{figure}[htb]
\begin{center}
  \includegraphics[width=120pt]{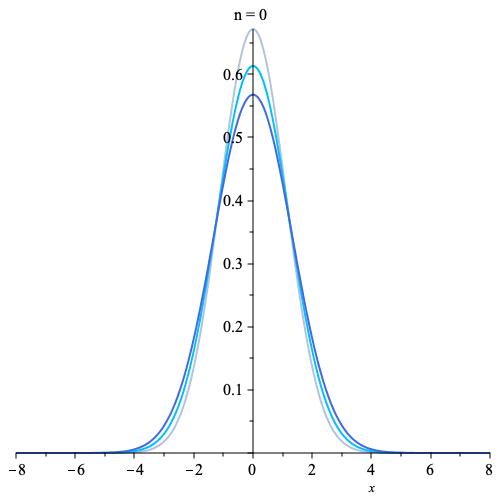}\hspace*{5pt}\includegraphics[width=120pt]{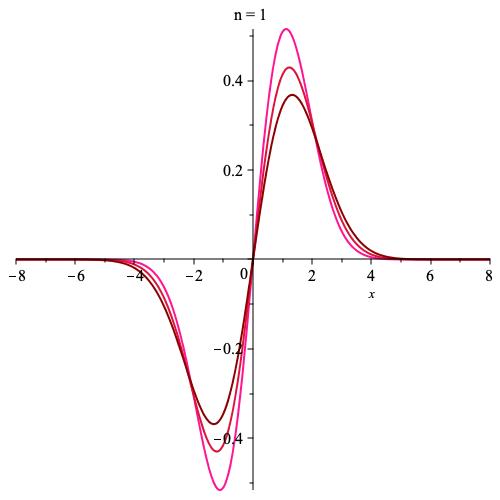}\hspace*{5pt}\includegraphics[width=120pt]{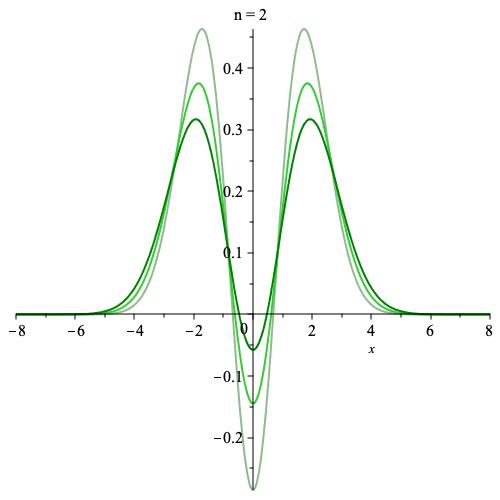}
  \includegraphics[width=120pt]{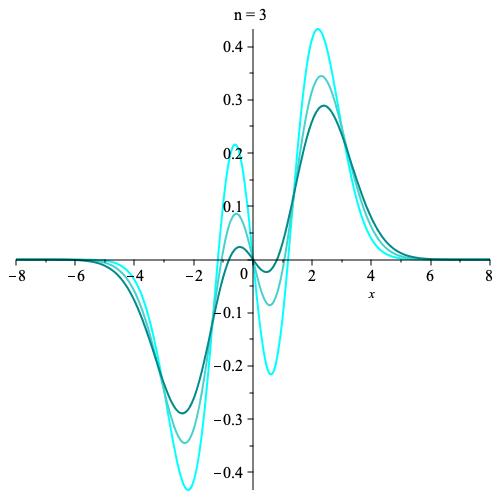}\hspace*{5pt}\includegraphics[width=120pt]{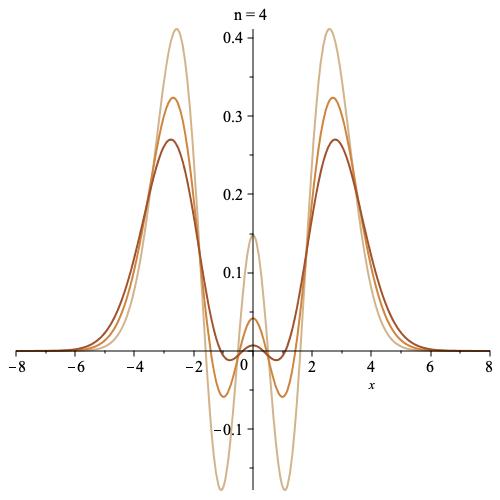}\hspace*{5pt}\includegraphics[width=120pt]{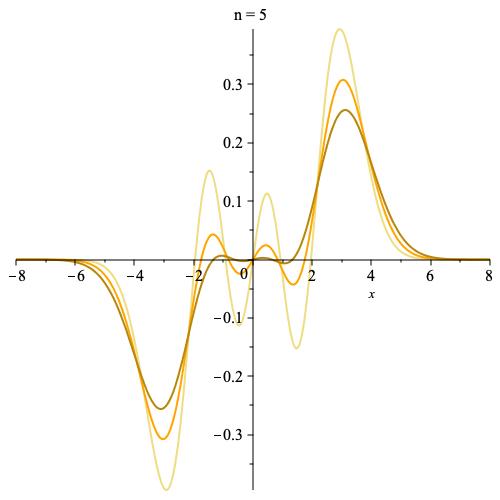}
\caption{The first six $\CC{H}^\infty$ Hermite-type functions $\varphi_n^{[\infty]}$ for $\sigma=\frac14,\frac12,\frac34$ in progressively darker hues.}
\label{fig:Hinf}
\end{center}
\end{figure}

Fig.~\ref{fig:Hinf} displays the functions $\varphi_n^{[\infty]}$, $n=0,\ldots,5$, for three different values of $\sigma\in(0,1)$. The zeros of $\varphi_n$ are scaled zeros of a Hermite polynomial and, the scaling being monotone, the zeros are  `squeezed' in a uniform manner for increasing $\sigma$, as evident in the figure.

\setcounter{equation}{0}
\setcounter{figure}{0}
\section{Bilateral Laguerre-type weights}

Deferring the standard Laguerre weight (which is not symmetric) to Section~7, we let $w(\xi)=(1+\xi^2)\ee^{-|\xi|}$, $\xi\in\BB{R}$. Note that the underlying orthogonal polynomials are unknown explicitly, yet can be computed.
The $\varphi_n^{[1]}$s are
\begin{Eqnarray*}
  \varphi_0^{[1]}(x)&=&\frac{2}{\sqrt{3}\sqrt{\pi}} \frac{1}{1+4x^2},\\
  \varphi_1^{[1]}(x)&=&\frac{16}{\sqrt{26}\sqrt{\pi}} \frac{x}{(1+4x^2)^2},\\
  \varphi_2^{[1]}(x)&=&\frac{2}{\sqrt{1167}\sqrt{\pi}} \frac{1+248x^2+208x^4}{(1+4x^2)^3},\\
  \varphi_3^{[1]}(x)&=&\frac{16}{\sqrt{23179}\sqrt{\pi}} \frac{-21x+456x^3+496x^5}{(1+4x^2)^4},\\
  \varphi_4^{[1]}(x)&=&\frac{2}{\sqrt{309347971}\sqrt{\pi}} \frac{2925-128784x^2+1703264x^4+3029760x^6+1369344x^8}{(1+4x^2)^5},\\
  \varphi_5^{[1]}(x)&=&\frac{16}{\sqrt{22678864934}\sqrt{\pi}} \frac{25749x\!-\!1017424x^3\!+\!5715040x^5\!+\!13510400x^7\!+\!7744768x^9}{(1+4x^2)^6}.
\end{Eqnarray*}
The general formula is a polynomial of degree $2n-[1-(-1)^n]/2$ in $x$ (of the same parity as $n$), divided by $(1+4x^2)^{n+1}$. This can be easily verified because by \R{eqn:phinformula} and $g(\xi)=\ee^{-|\xi|/2}$ each $\varphi_n^{[1]}$ is a linear combination of $\lambda_n,\lambda_{n-2},\ldots$, where
\begin{displaymath}
  \lambda_n(x)=\frac{\ii^n}{\sqrt{2\pi}} \int_{-\infty}^\infty \xi^n \ee^{-|\xi|^2+\ii x\xi}\D\xi,\qquad n\in\BB{Z}_+
\end{displaymath}
and $\lambda_n'(x)=\lambda_{n+1}(x)$ implies that
\begin{displaymath}
  \lambda_n(x)=\lambda^{(n)}_0(x)=\frac{2\sqrt{2}}{\sqrt{\pi}} \frac{\D^n}{\D x^n} \frac{1}{1+4x^2},\qquad n\in\BB{Z}_+.
\end{displaymath}

\setcounter{equation}{0}
\setcounter{figure}{0}
\section{Bessel-like functions}

\subsection{Transformation of Chebyshev polynomials}

We rewrite  \R{eq:ChebyIntegral} in the form \R{eqn:phinformula}, 
\begin{displaymath}
  \varphi_n(x)=\frac{\ii^n}{\sqrt{2\pi}} \int_{-1}^1 \tilde{\CC{T}}_n(\xi) \ee^{\ii x\xi} \frac{\D\xi}{\sqrt{1-\xi^2}} =(-1)^n \CC{J}_n(x),
\end{displaymath}
where $\tilde{\CC{T}}_n=\sqrt{2/\pi}\CC{T}_n$ (except that $\tilde{\CC{T}}_0\equiv \CC{T}_0/\sqrt{\pi}$) are orthonormal Chebyshev polynomials. This does not fit within the framework of Theorem~2, because   $v(\xi)=(1-\xi^2)^{-3/2}\chi_{(-1,1)}(\xi)$ (derived using \R{eqn:phinformula}) is not $\CC{L}_2$. Moreover, it is easy to verify directly that the $\varphi_n$s cannot be bounded in {\em any\/} Sobolev norm because the Weber--Schafheitlin formula \cite[10.22.57]{dlmf} implies that for $\Re\lambda>0$
\begin{displaymath}
  \int_{-\infty}^\infty \frac{\varphi_n^2(x)\D x}{x^\lambda}=\int_{-\infty}^\infty \frac{\CC{J}_n^2(x)\D x}{x^\lambda}=\frac{\CC{\Gamma}(n+\frac12)\CC{\Gamma}(\lambda)}{2^{\lambda-1} \CC{\Gamma}^2(\frac12\lambda+\frac12) \CC{\Gamma}(\frac12\lambda+n+\frac12)}\stackrel{\lambda\rightarrow0}{\longrightarrow}\infty.
\end{displaymath}

If instead of a Chebyshev measure we use the Legendre measure, $w(\xi)=\chi_{(-1,1)}(\xi)$, the state of affairs is different: $g(\xi)=\chi_{(-1,1)}$ results in
\begin{equation}
  \label{eq:Besselj}
  \varphi_n(x)=(-1)^n\sqrt{\frac{n+\frac12}{x}} \CC{J}_{n+\frac12}(x),\qquad x\in\BB{R},
\end{equation}
and the $\varphi_n$s are integrable on $\BB{R}$. 

\subsection{The Legendre weight}

The most obvious example of an $\CC{H}^1(\BB{R})$ system is based on the Legendre weight $w(\xi)=\chi_{(-1,1)}(\xi)$, in which the orthonormal polynomials are $p_n=\sqrt{n+\frac12}\CC{P}_n$. Thus,
\begin{displaymath}
  \varphi_n^{\langle1\rangle}(x)=\frac{\ii^n}{\sqrt{2\pi}}\sqrt{n+\frac12} \int_{-1}^1 \frac{\CC{P}_n(\xi)}{\sqrt{1+\xi^2}} \ee^{\ii x\xi}\D\xi,\qquad n\in\BB{Z}_+.
\end{displaymath}

While $\{\varphi_n^{\langle1\rangle}\}_{n\in\bb{Z}_+}$ is orthonormal in $\CC{H}^1(\BB{R})$, it is not a complete basis because all Fourier spectra are restricted to $[-1,1]$, yet it might be of an independent interest. Perhaps more vexing issue is that above integrals are not available in an explicit form. This is not an insurmountable problem in the computation of the $\varphi_n^{\langle1\rangle}$s which we can compute for individual values of $x$ using a Fast Fourier Transform \cite{townsend2018fast,olver20fau}, although it presents an obvious obstacle to their analysis. 

\begin{figure}[htb]
\begin{center}
  \includegraphics[width=120pt]{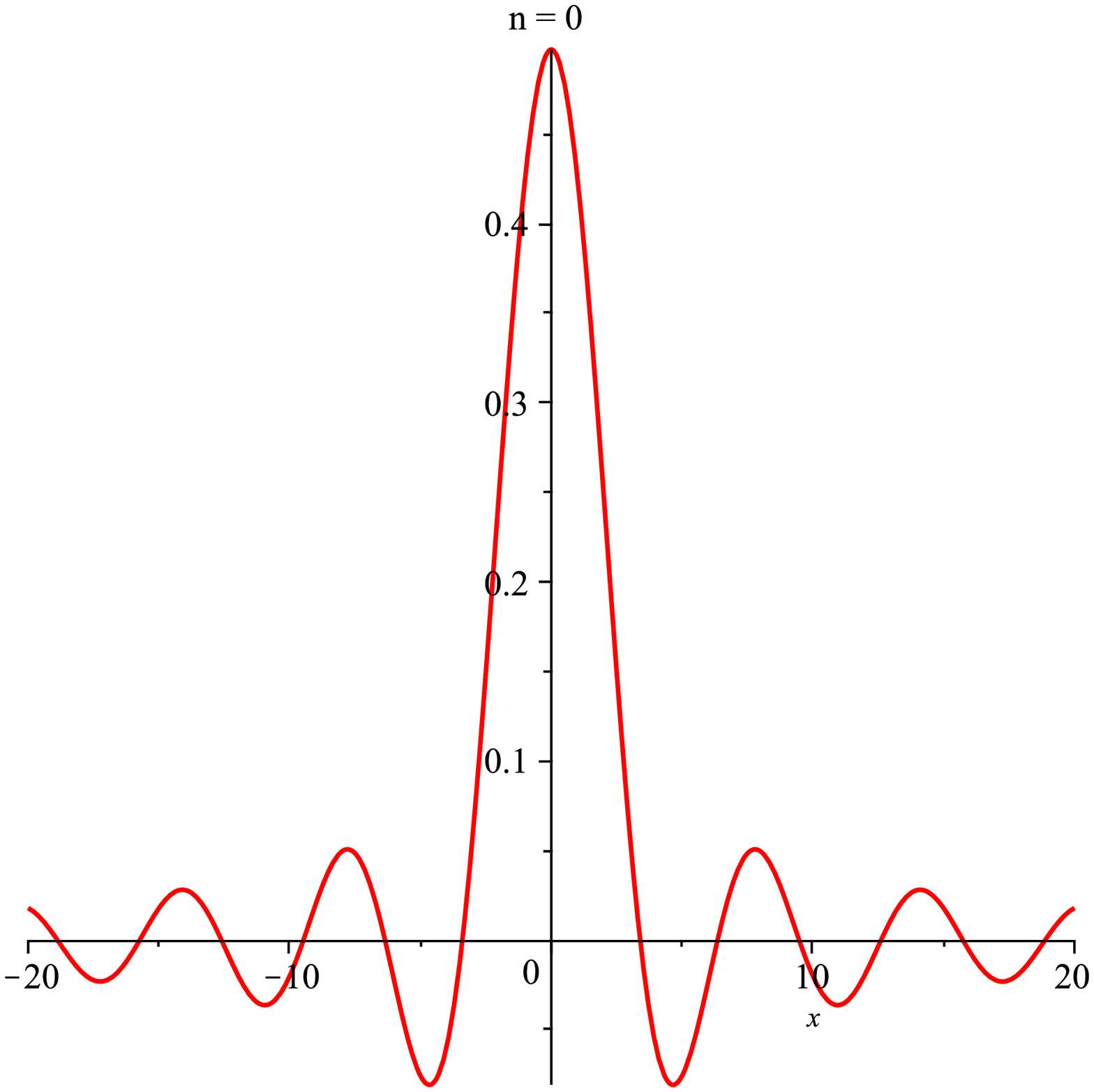}\hspace*{5pt}\includegraphics[width=120pt]{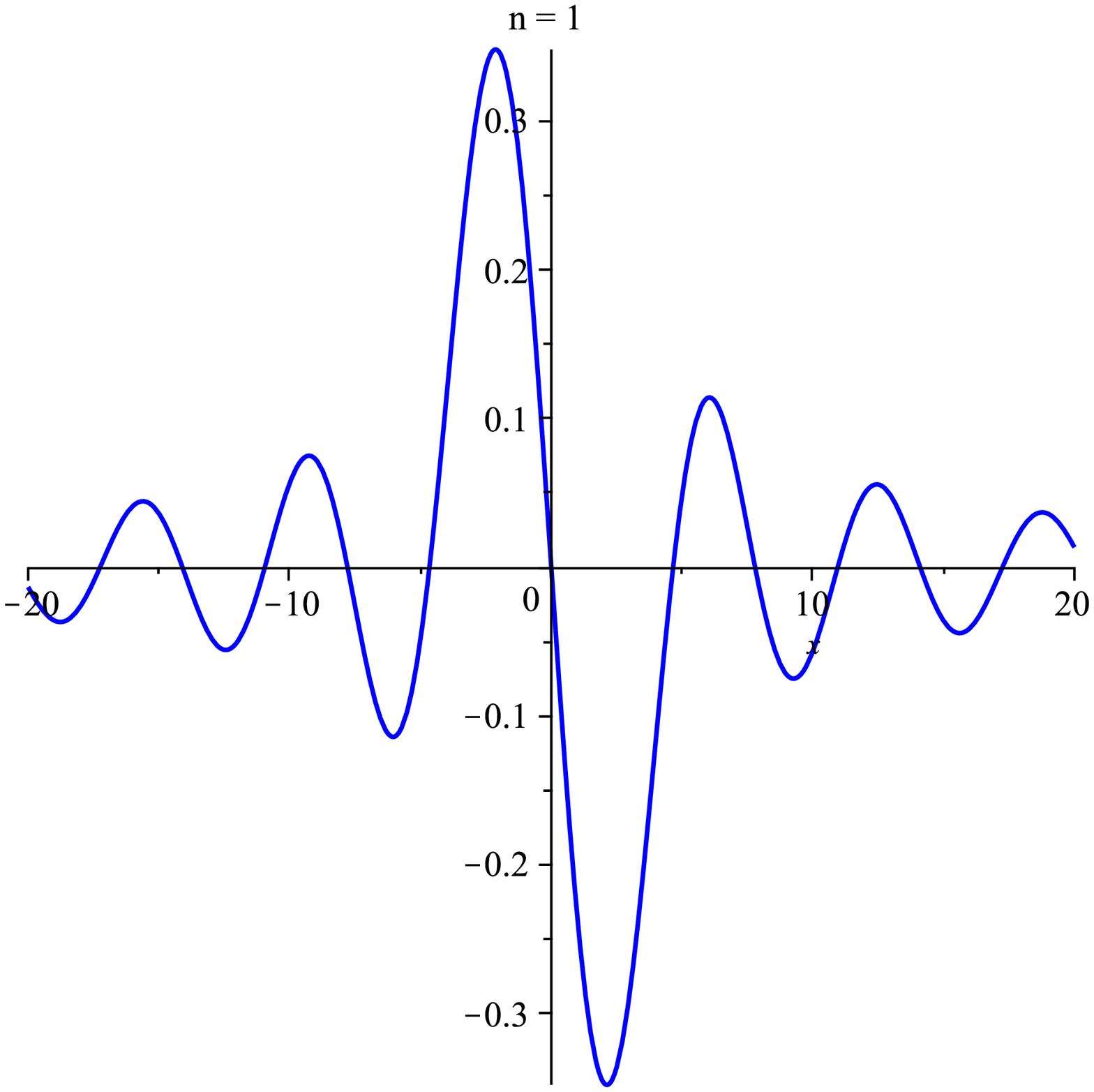}\hspace*{5pt}\includegraphics[width=120pt]{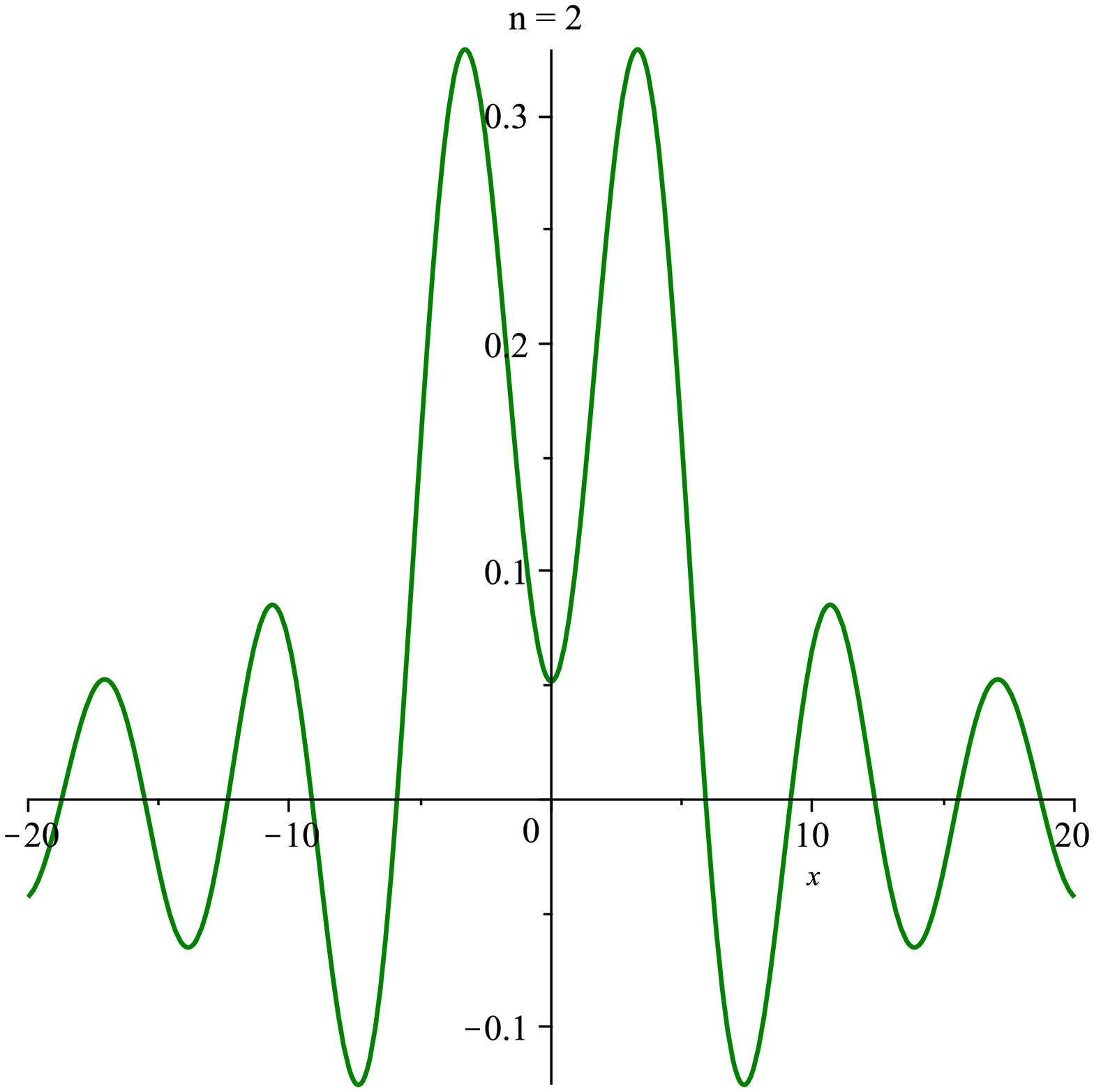}
  \includegraphics[width=120pt]{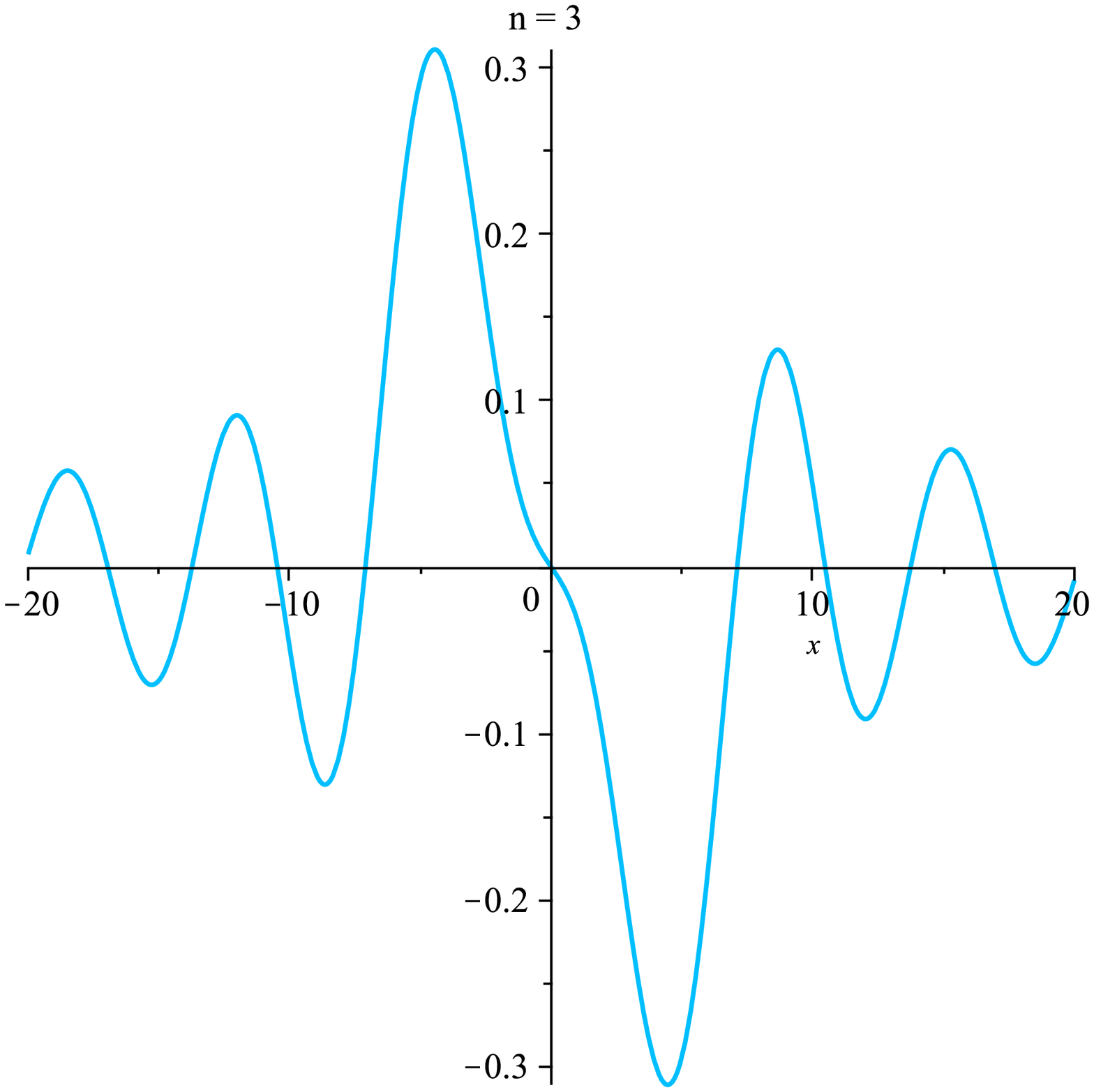}\hspace*{5pt}\includegraphics[width=120pt]{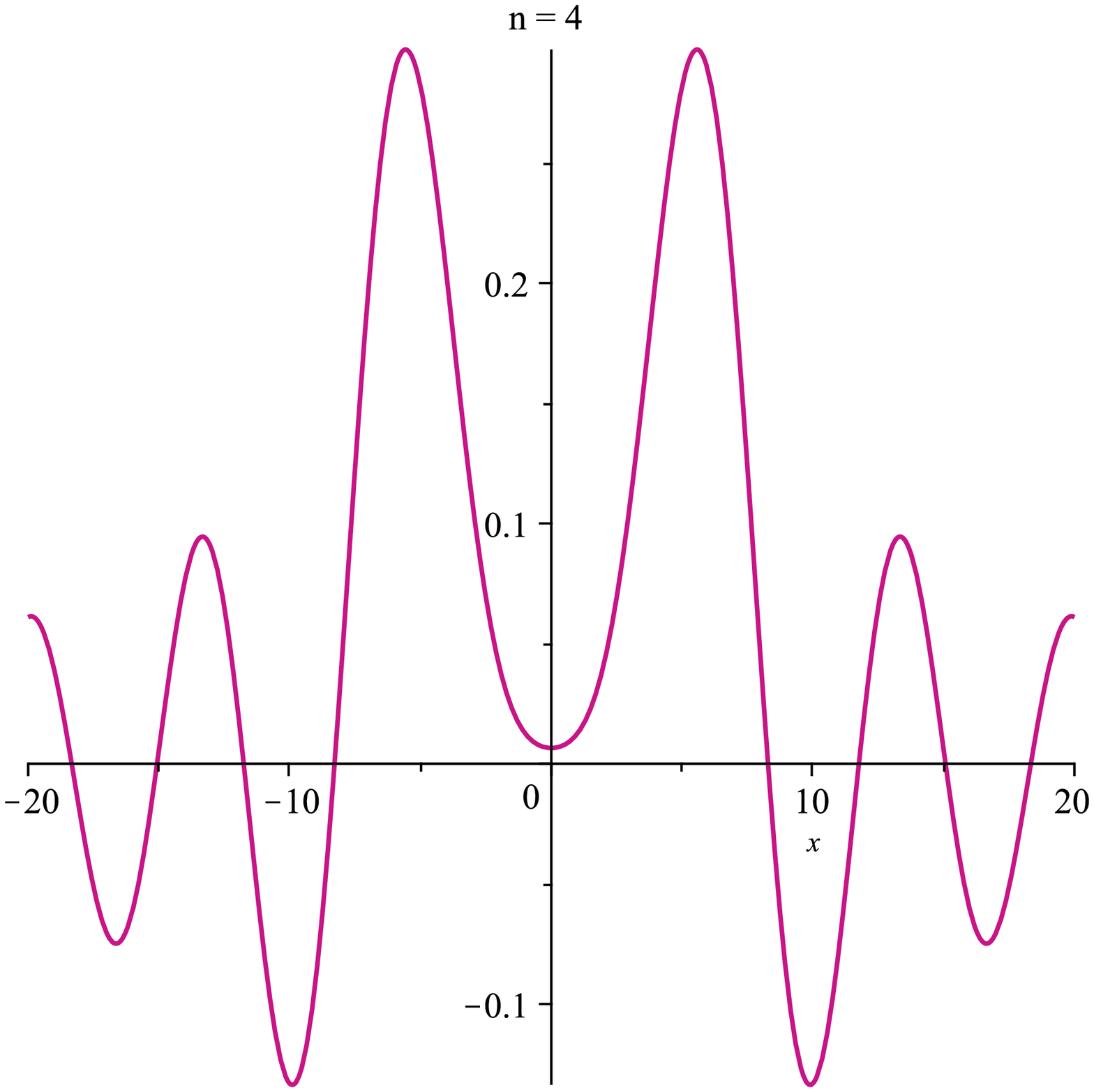}\hspace*{5pt}\includegraphics[width=120pt]{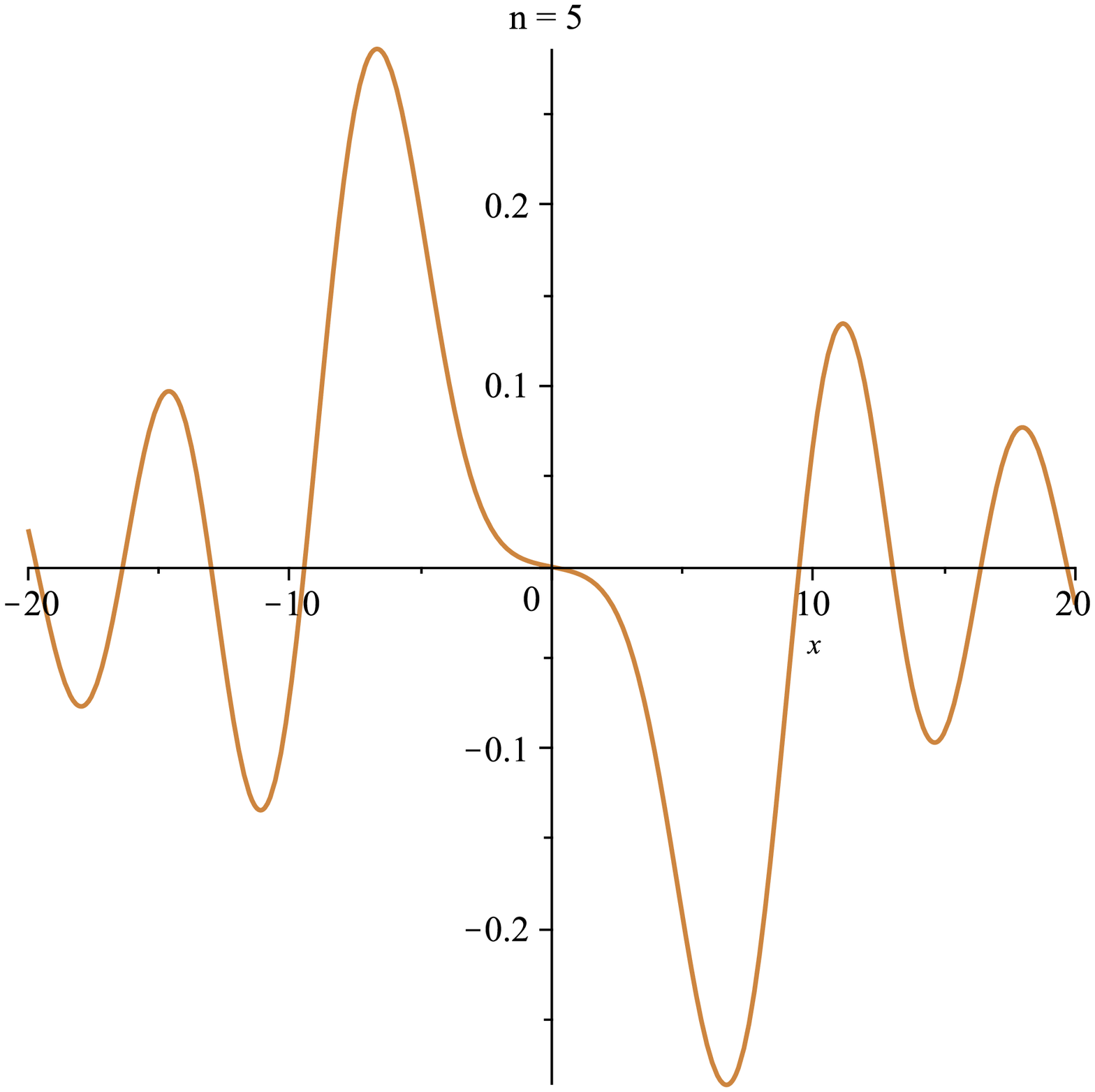}
\caption{The first six functions $\varphi_n^{\langle1\rangle}$ for the Legendre weight.}
\label{fig:3.3}
\end{center}
\end{figure}

In Fig.~\ref{fig:3.3} we have computed $\varphi_0^{\langle1\rangle},\ldots,\varphi_5^{\langle1\rangle}$ numerically. Like other transformed functions \R{eq:DefVarphi} or \R{eqn:phinformula}, the $\varphi_n^{\langle1\rangle}$s seem to be endowed with a wealth of structural features and regularities which have been discussed briefly (for \R{eq:DefVarphi}) in \cite{Iserles21daf} but overall are a subject for future research.

\subsection{Sobolev--Legendre cascades}

We revisit the essence of Subsections 3.1.1--2, except that the range of integration is now $[-1,1]$.  Firstly, we  set $w=\chi_{(-1,1)}$, let $p_n$ be the (orthonormalised) Legendre polynomials and, for every $s\in\BB{Z}_+$ set 
\begin{equation}
  \label{eq:SL1st}
  \varphi_n^{\langle s \rangle}(x)=\frac{\ii^n}{\sqrt{2\pi}} \int_{-1}^1 p_n(\xi) \left(\sum_{\ell=0}^s \xi^{2\ell}\!\right)^{\!-1/2} \ee^{\ii x\xi}\D\xi,\qquad n\in\BB{Z}_+.
\end{equation}
Secondly, we might define $w_s(\xi)=\chi_{(-1,1)}(\xi)\sum_{\ell=0}^s \xi^{2\ell}$, $s\in\BB{Z}_+$, and set 
\begin{equation}
  \label{eq:SL2nd}
  \varphi_n^{[s]}(x)=\frac{\ii^n}{\sqrt{2\pi}} \int_{-1}^1 p^{[s]}_n(\xi) \ee^{\ii x\xi}\D\xi,\qquad n\in\BB{Z}_+,
\end{equation}
where $\{p_n^{[s]}\}_{n\in\bb{Z}_+}$ is the orthonormal polynomial system corresponding to the weight $w_s$. It follows at once from Theorem~2 that
\begin{displaymath}
  \sum_{\ell=0}^s \int_{-\infty}^\infty \frac{\D^\ell \varphi_m^{\langle s \rangle}(x)}{\D x^\ell} \frac{\D^\ell \varphi_n^{\langle s \rangle}(x)}{\D x^\ell} \D x=\sum_{\ell=0}^s \int_{-\infty}^\infty \frac{\D^\ell \varphi_m^{[s]}(x)}{\D x^\ell} \frac{\D^\ell \varphi_n^{[s]}(x)}{\D x^\ell} \D x=\delta_{m,n}
\end{displaymath}
for all $m,n\in\BB{Z}_+$ and both $\{\varphi_n^{\langle s \rangle}\}_{n\in\bb{Z}_+}$ and $\{\varphi_n^{[s]}\}_{n\in\bb{Z}_+}$ are orthonormal sets in $\CC{H}^s(\BB{R})$. Of course, neither is dense in the Sobolev space because their Fourier spectra are restricted to $[-1,1]$ -- they are dense in an obvious generalisation of Paley--Wiener spaces to the realm of Sobolev spaces. The systems \R{eq:SL1st} and \R{eq:SL2nd} are the {\em Sobolev--Legendre cascades\/} of {\em the first\/} and {\em the second kind,\/} respectively. 

We recall a major practical difference between the two cascades: except for the case $s=0$ (when $\varphi^{\langle0\rangle}_n=\varphi^{[0]}_n$ has been given in \R{eq:Besselj}), $\varphi_n^{\langle s \rangle}$ is unknown in an explicit form while $\varphi_n^{[s]}$, being an integral in $[-1,1]$ of a polynomial times $\ee^{\ii x\xi}$, can be computed at great ease and is a linear combination of spherical Bessel functions. Consequently, in the sequel we focus on the Sobolev--Lagendre cascade of the second kind. 

\begin{figure}[htb]
\begin{center}
  \includegraphics[width=120pt]{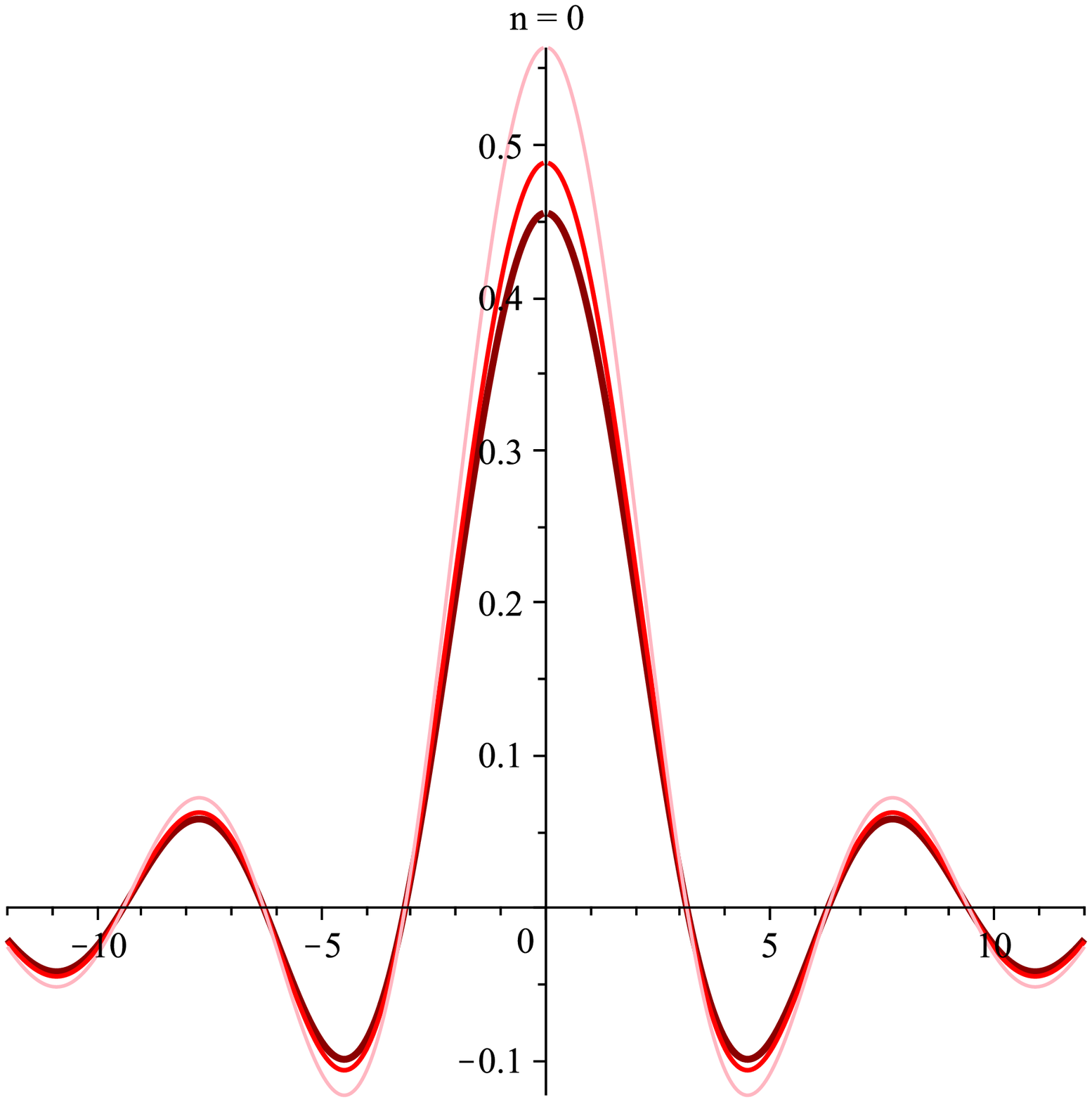}\hspace*{5pt}\includegraphics[width=120pt]{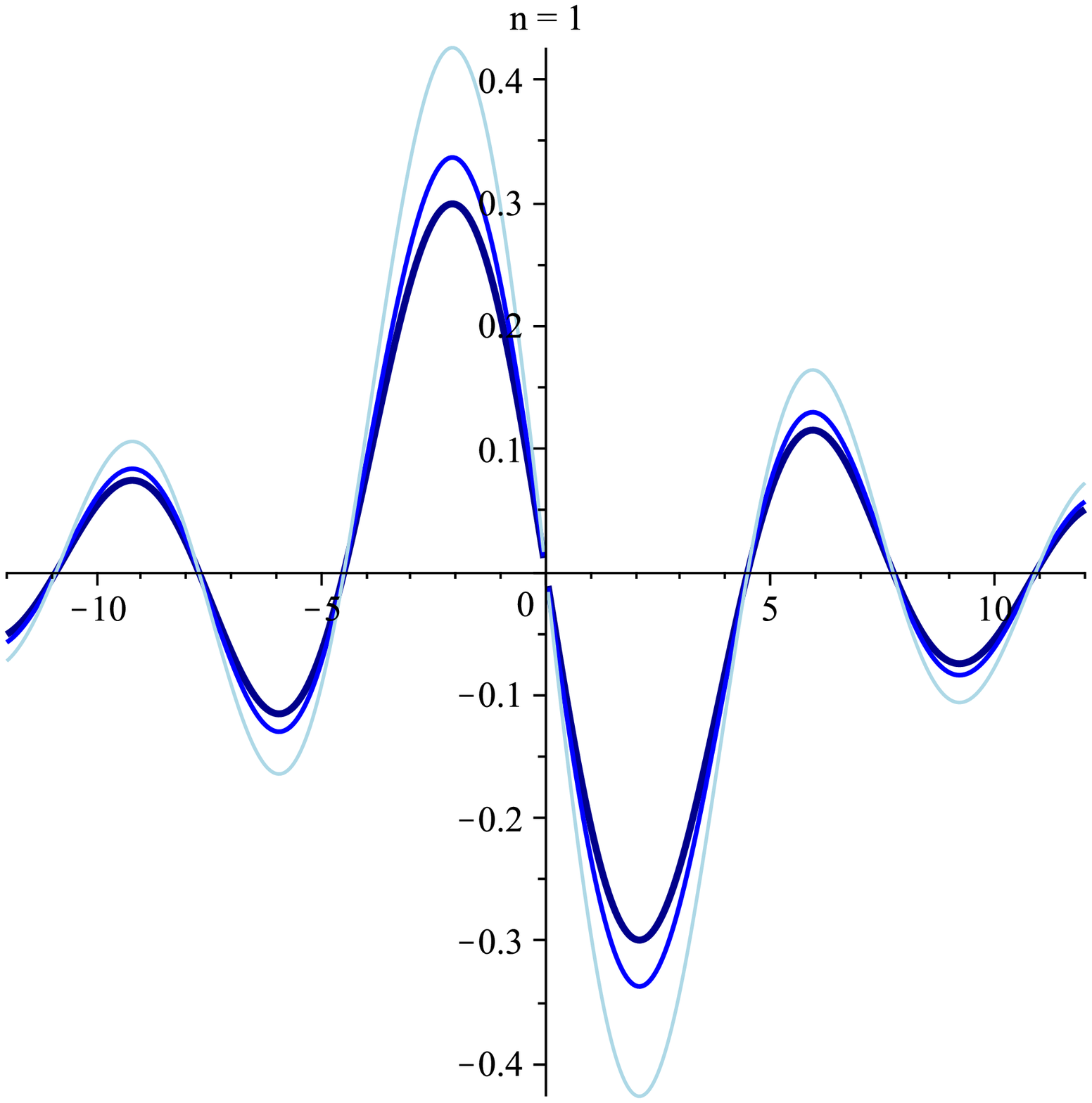}\hspace*{5pt}\includegraphics[width=120pt]{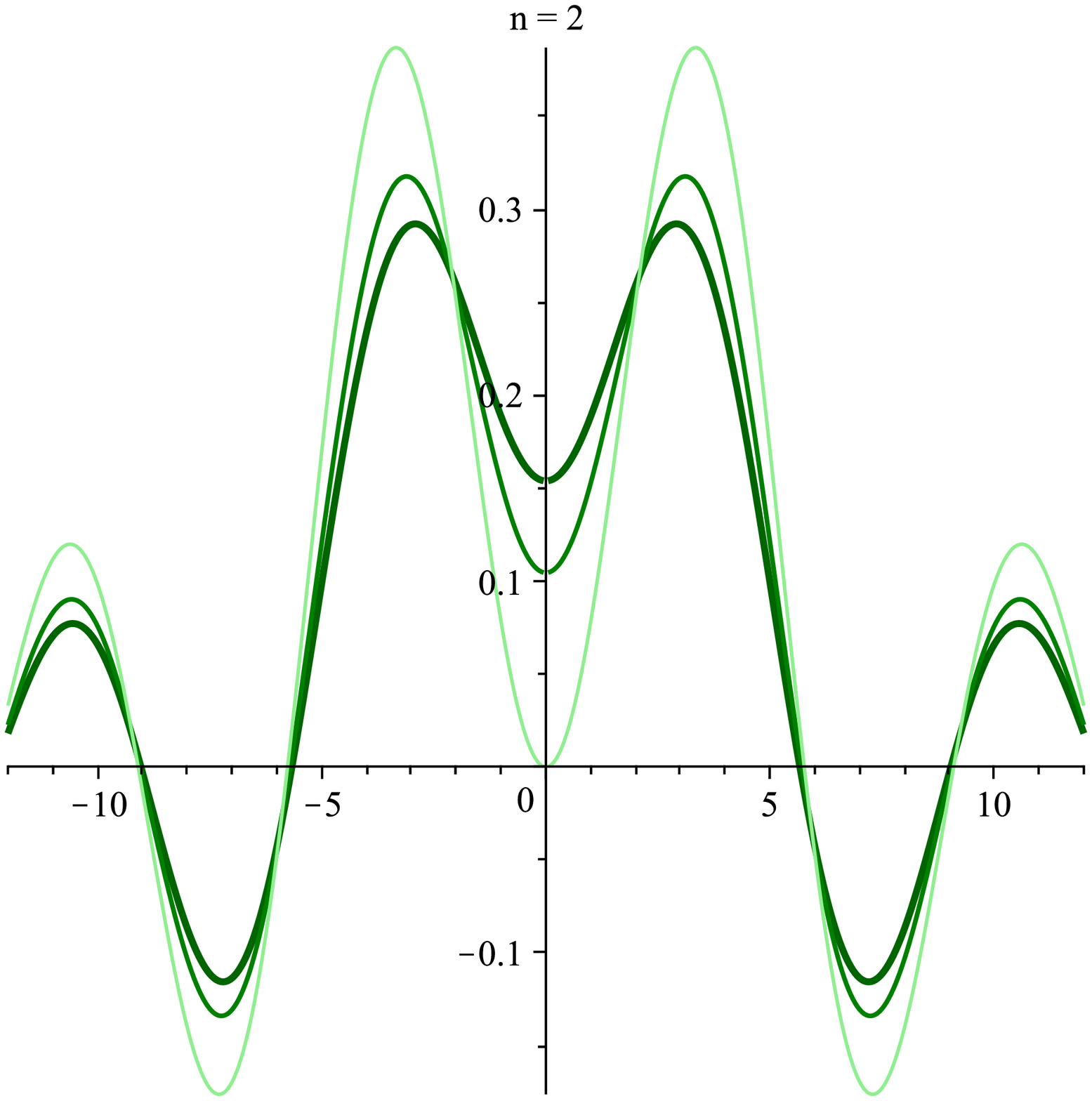}
  \includegraphics[width=120pt]{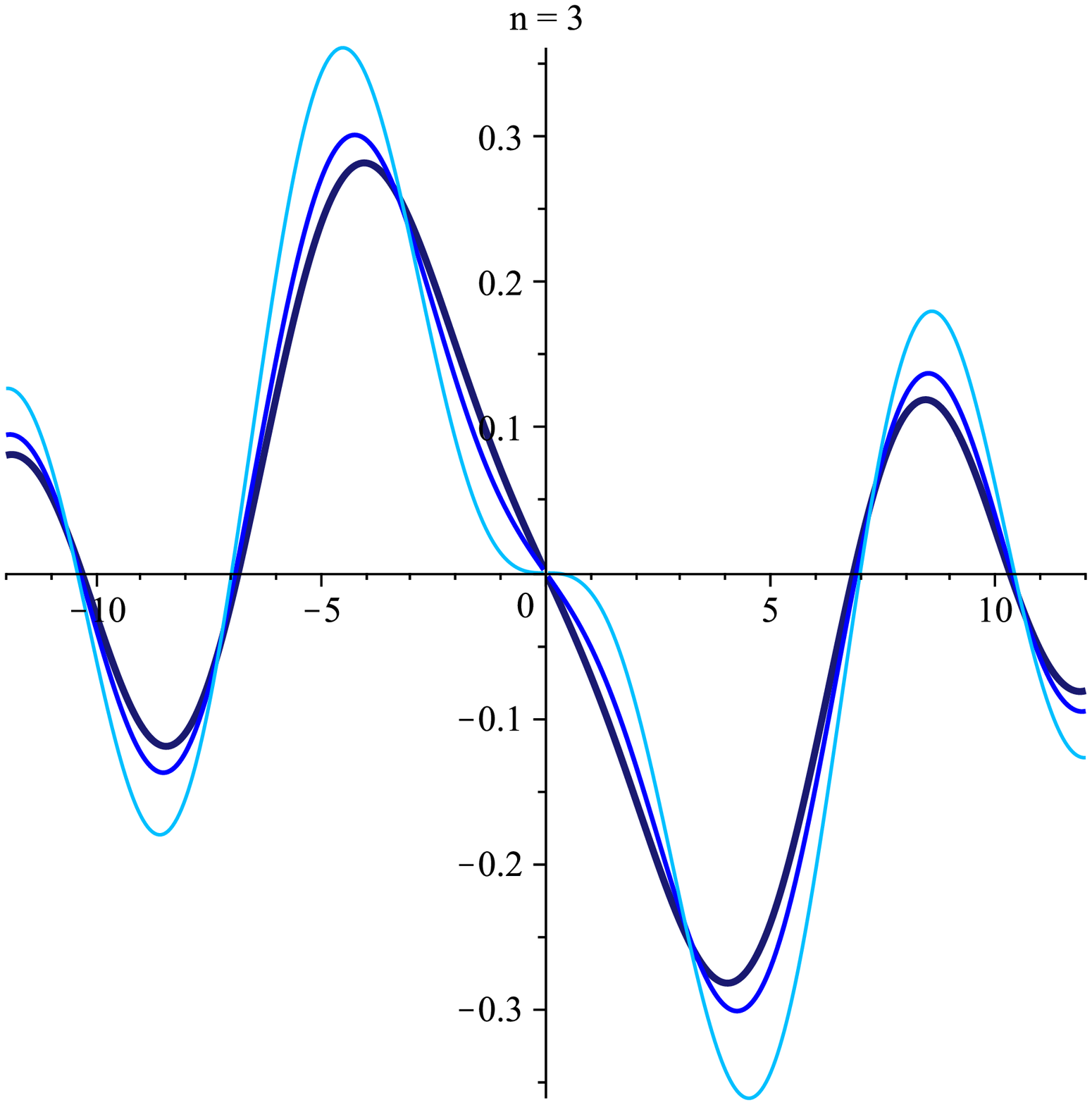}\hspace*{5pt}\includegraphics[width=120pt]{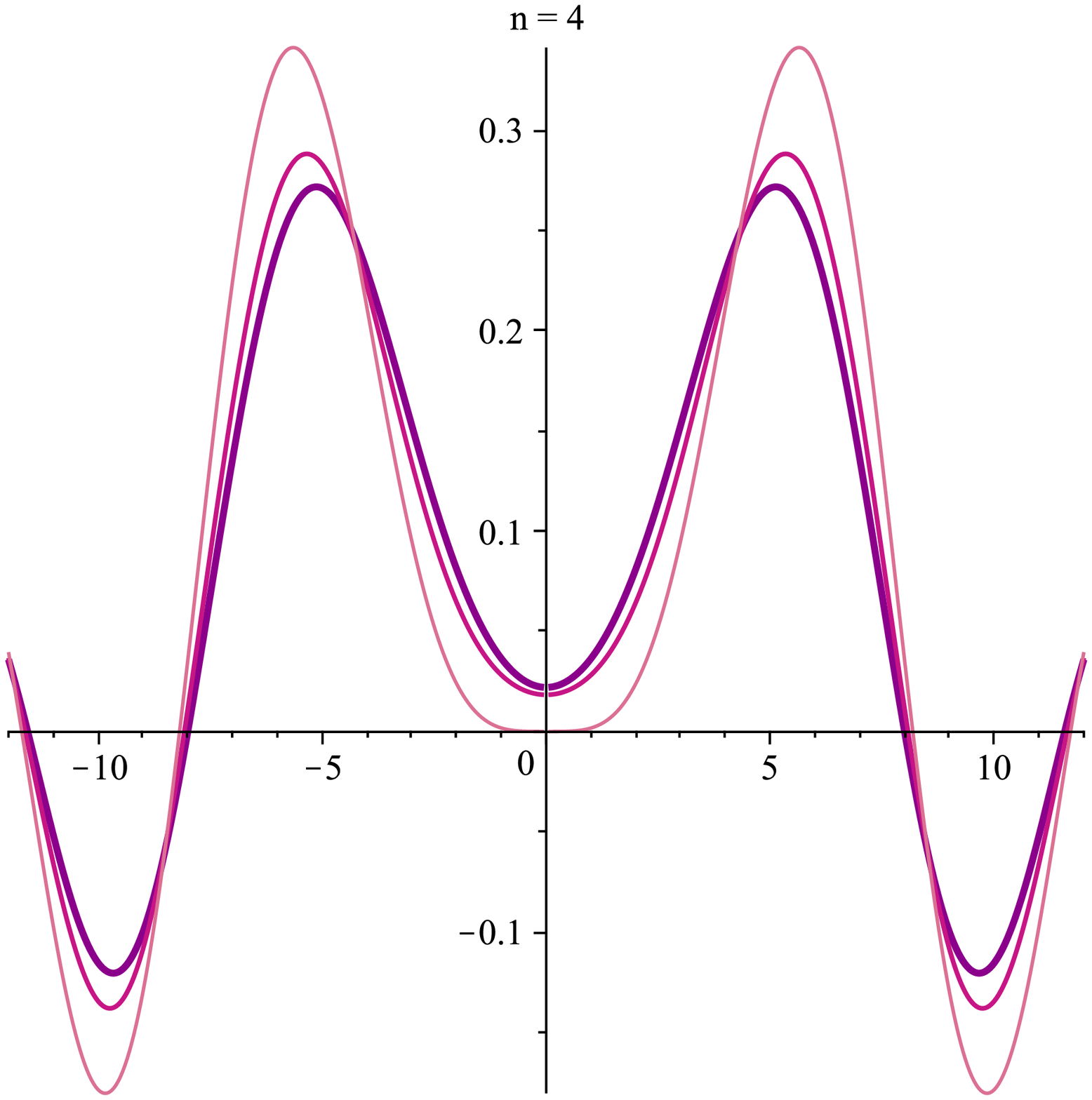}\hspace*{5pt}\includegraphics[width=120pt]{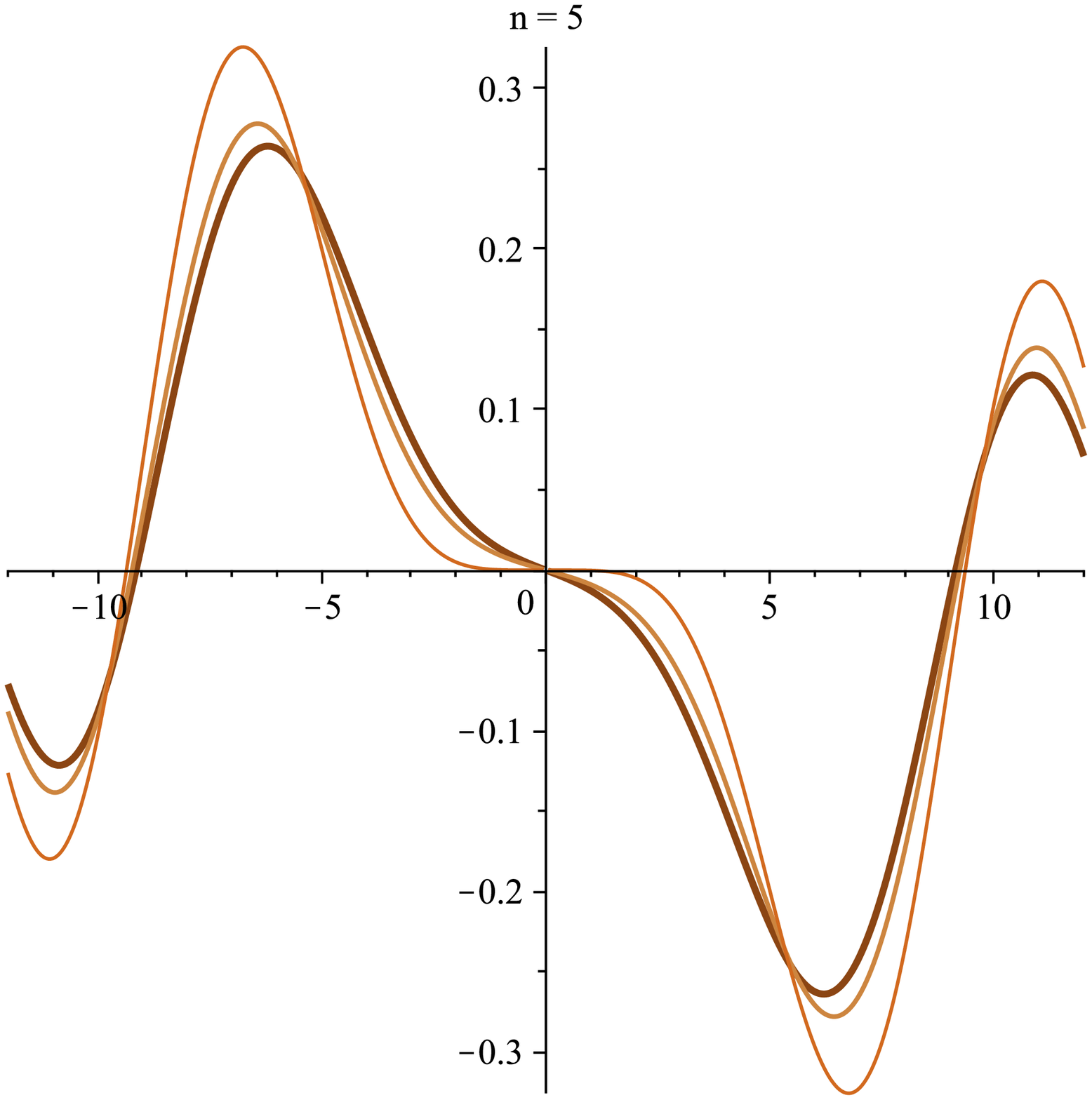}
\caption{The Sobolev--Legendre cascade of the second kind: The first six functions $\varphi_n^{[s]}$ for $s=0,1,2$: increasing $s$ corresponds to increasing line thickness and darker hue.}
\label{fig:3.4}
\end{center}
\end{figure}

Fig.~\ref{fig:3.4} displays the beginning (that is, $s=0,1,2$) of the cascade of the second kind. The obvious observation is that $\varphi^{[s]}_0$ is a scalar multiple of $\varphi^{[0]}_0$ and the same is true for $\varphi_1^{[s]}$ and $\varphi_1^{[0]}$, respectively. This follows from \R{eq:SL2nd} because  $p_0^{[s]}$ is a constant, while $p_1^{[s]}$ is a scalar multiple of $\xi$. Another obvious indication is that, as $s$ grows, $\varphi_n^{[s]}$ converges pointwise to a function $\varphi_n^{[\infty]}$ yet this might be less interesting than it appears. In particular,

\begin{lemma}
 $\varphi_n^{[\infty]}\equiv0$.
\end{lemma}

\begin{proof}
 Let 
 \begin{displaymath}
  u_s=\int_{-1}^1 \sum_{\ell=0}^s \xi^{2\ell}\D\xi,\qquad s\in\BB{Z}_+.
 \end{displaymath}
  Then $p_0^{[0]}\equiv 1/\sqrt{u_2}$ and, by \R{eq:SL2nd}, 
  \begin{displaymath}
    \varphi_0^{[s]}(x)=\sqrt{\frac{2}{\pi u_s}} \frac{\sin x}{x}.
  \end{displaymath}
  But 
  \begin{displaymath}
    u_s=\sum_{\ell=0}^s \frac{1}{\ell+\frac12}\stackrel{\ell\rightarrow\infty}{\longrightarrow}\infty
  \end{displaymath}
  and the lemma follows. 
\end{proof}

Alternatively, $\lim_{s\rightarrow\infty} w^{[s]}(\xi)=(1-\xi^2)^{-1}\chi_{(-1,1)}\not\in\CC{L}_2(\BB{R})$. An obvious remedy, which we do not pursue in this paper, is to consider the weight $w_s(\xi)=\sum_{\ell=0}^s \sigma^\ell\xi^{2\ell}$ for some $\sigma\in(0,1)$.

\subsection{The Sobolev-ultraspherical cascade of the second kind}

We construct a cascade of the second kind based on the ultraspherical weight $w^{[0]}(\xi)=(1-\xi^2)\chi_{(-1,1)}(\xi)$. Therefore
\begin{displaymath}
  w^{[s]}(\xi)=w^{[0]}(\xi)\sum_{\ell=0}^s \xi^{2\ell}=(1-\xi^{2s+2})\chi_{(-1,1)},\qquad s\in\BB{Z}_+
\end{displaymath}
and, as $s\rightarrow\infty$,  $w^{[s]}$ converges weakly to the Legendre weight.

\begin{figure}[htb]
\begin{center}
  \includegraphics[width=120pt]{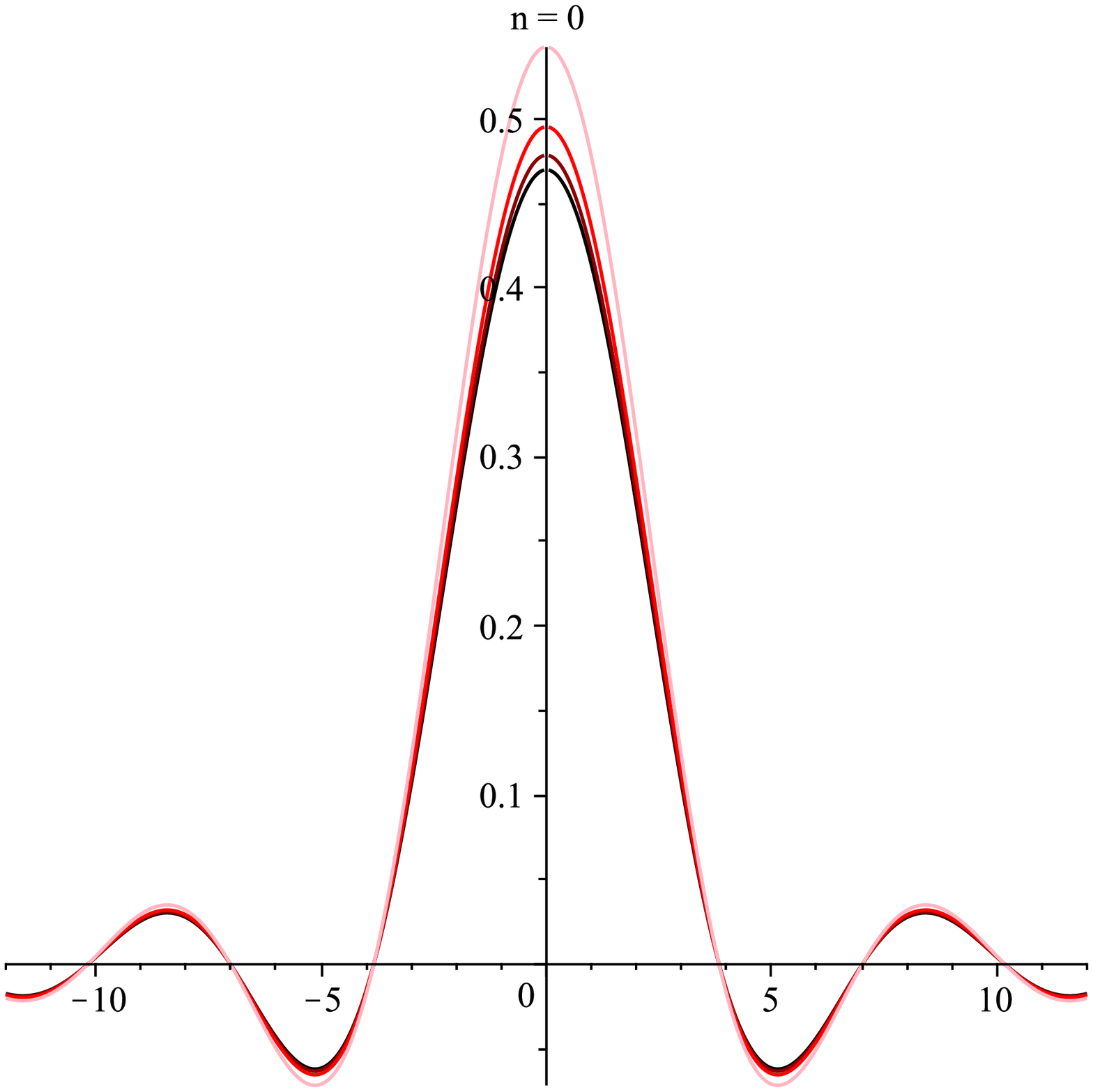}\hspace*{5pt}\includegraphics[width=120pt]{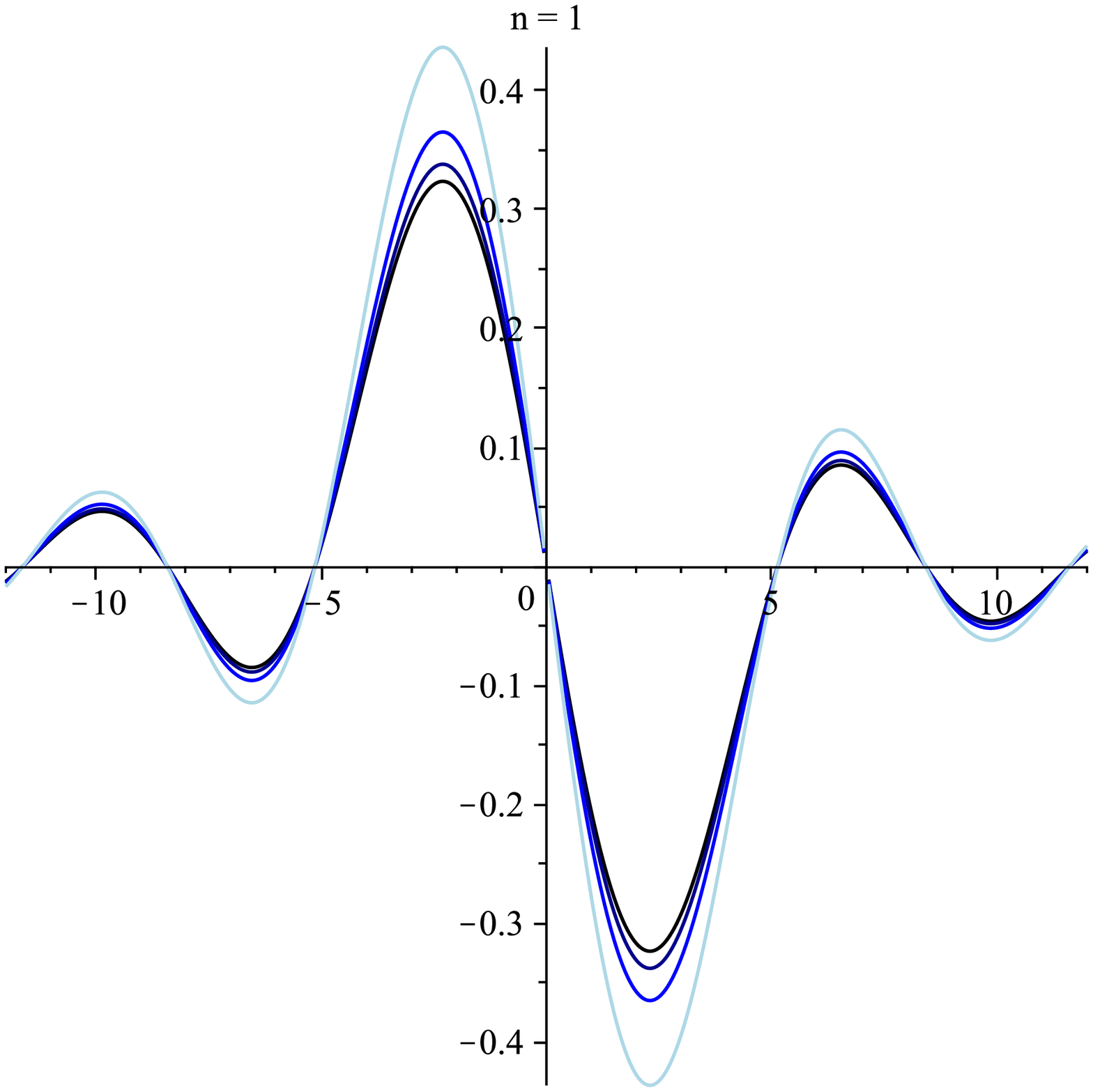}\hspace*{5pt}\includegraphics[width=120pt]{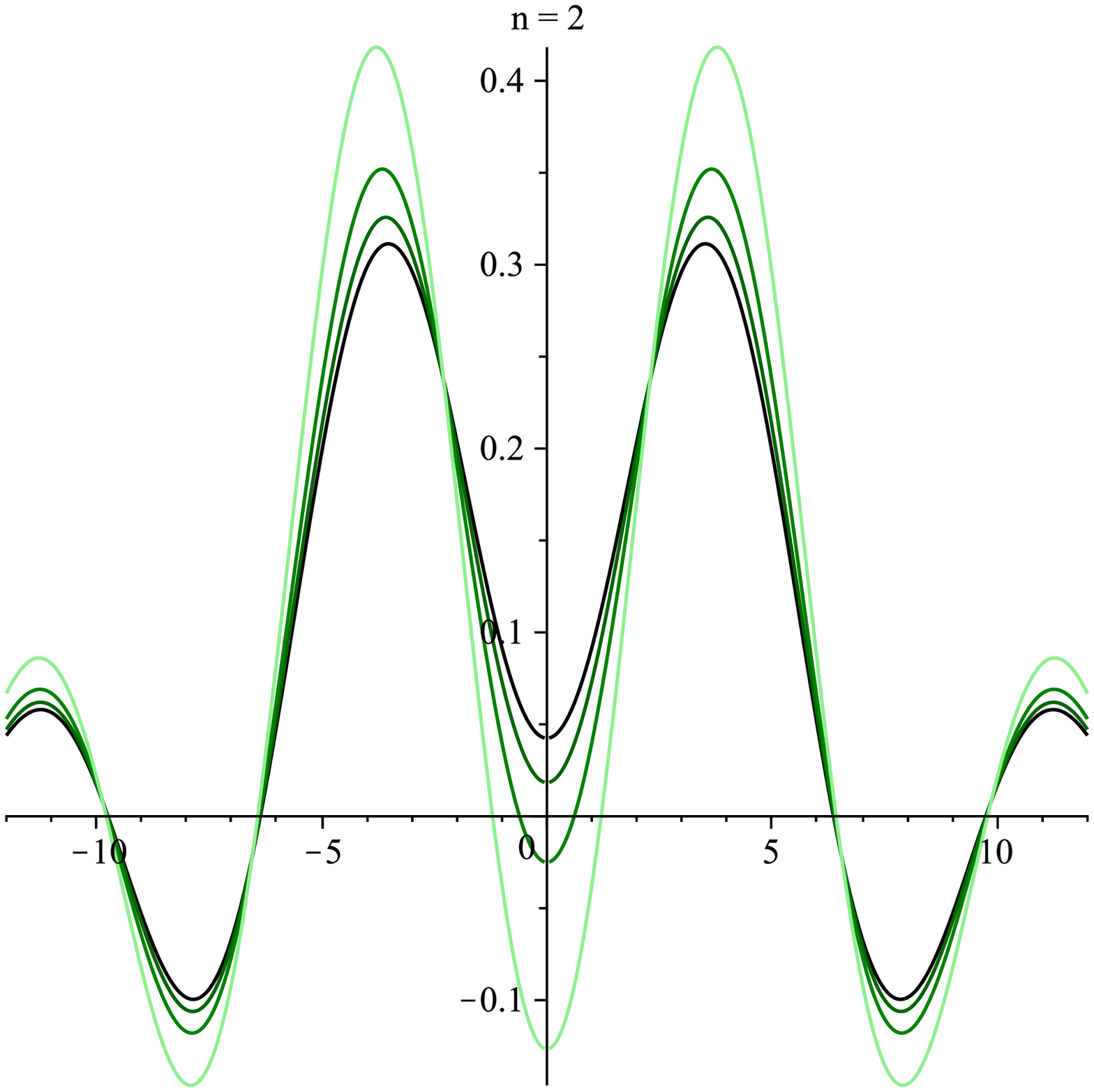}
  \includegraphics[width=120pt]{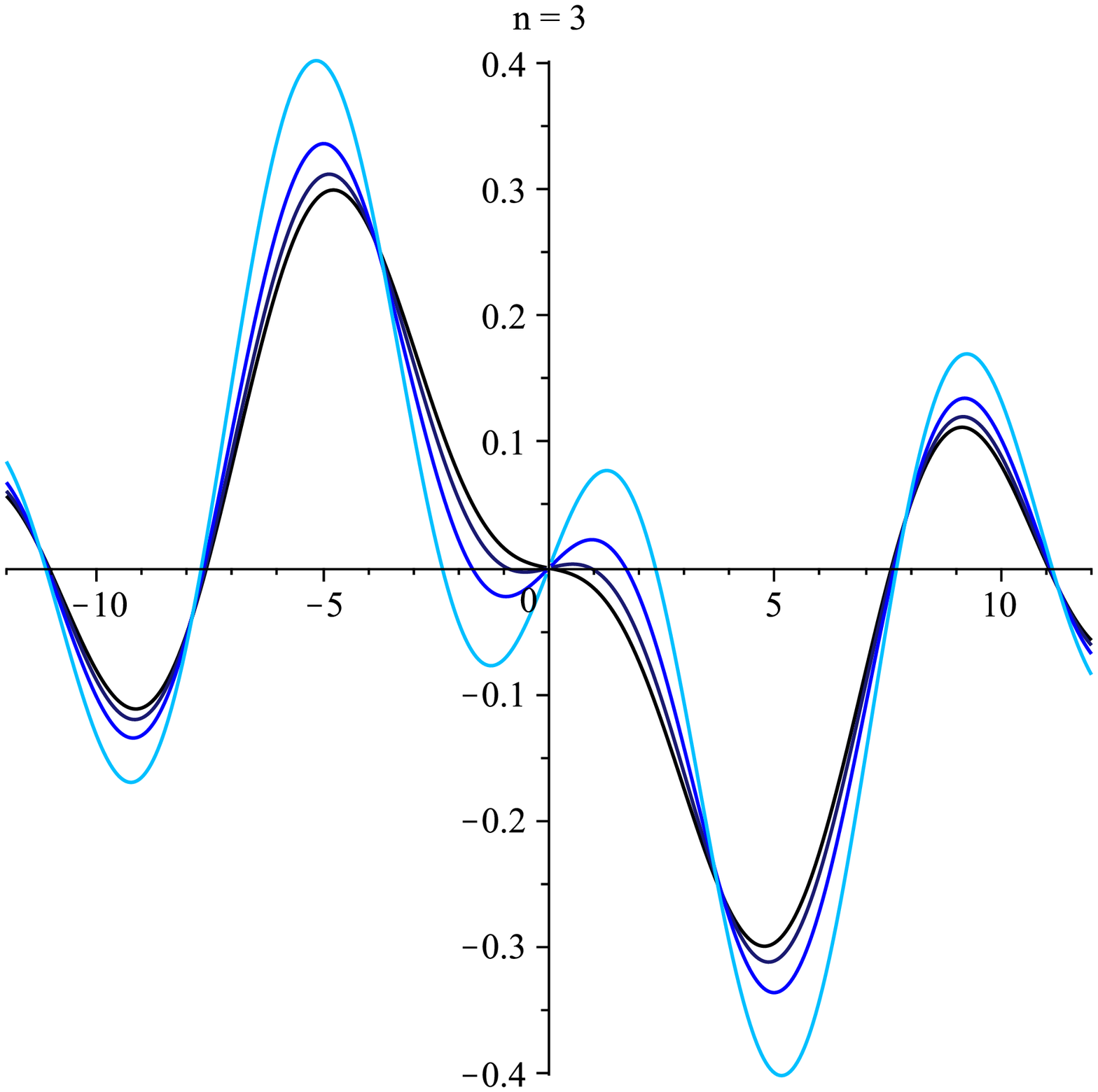}\hspace*{5pt}\includegraphics[width=120pt]{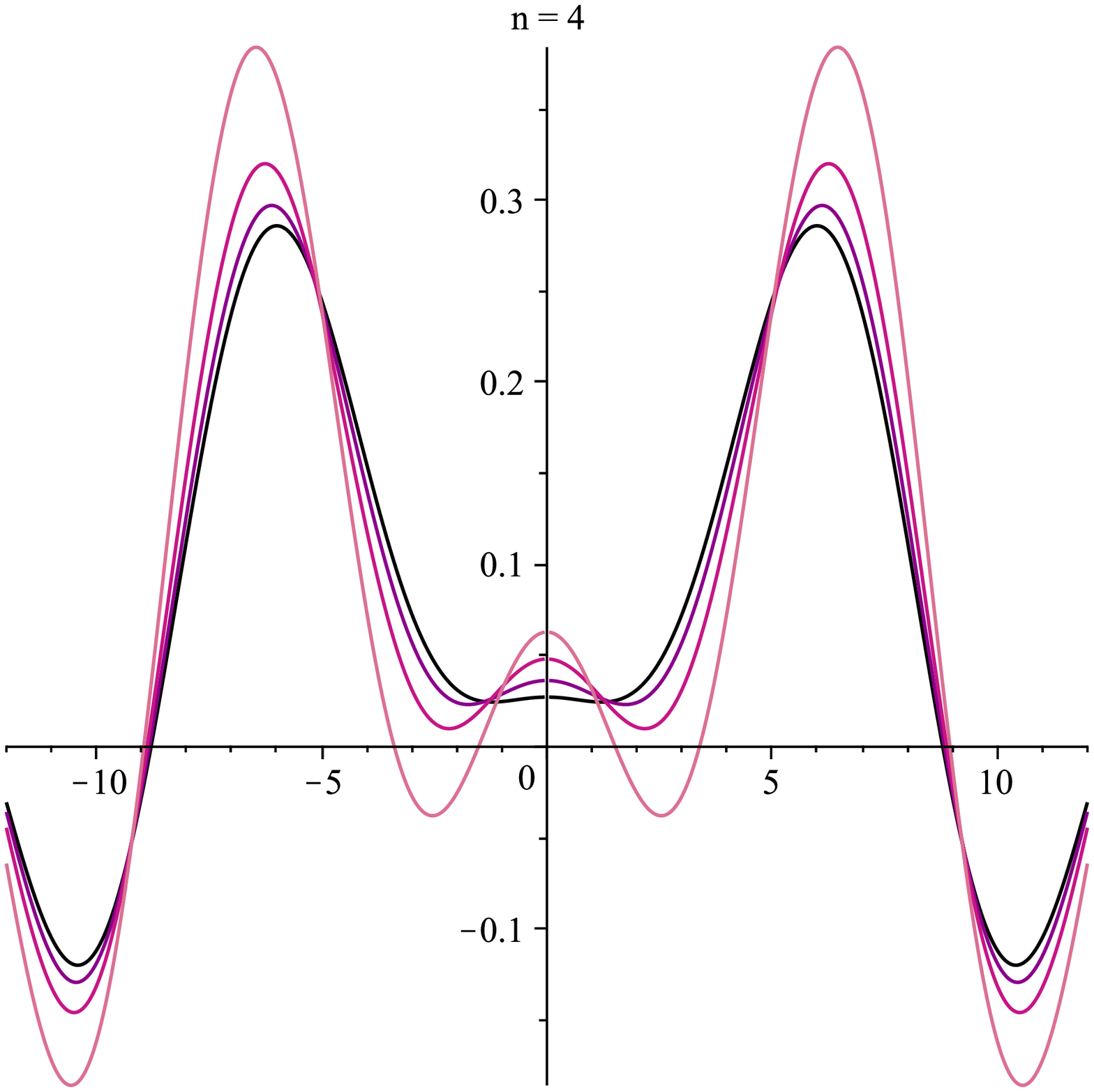}\hspace*{5pt}\includegraphics[width=120pt]{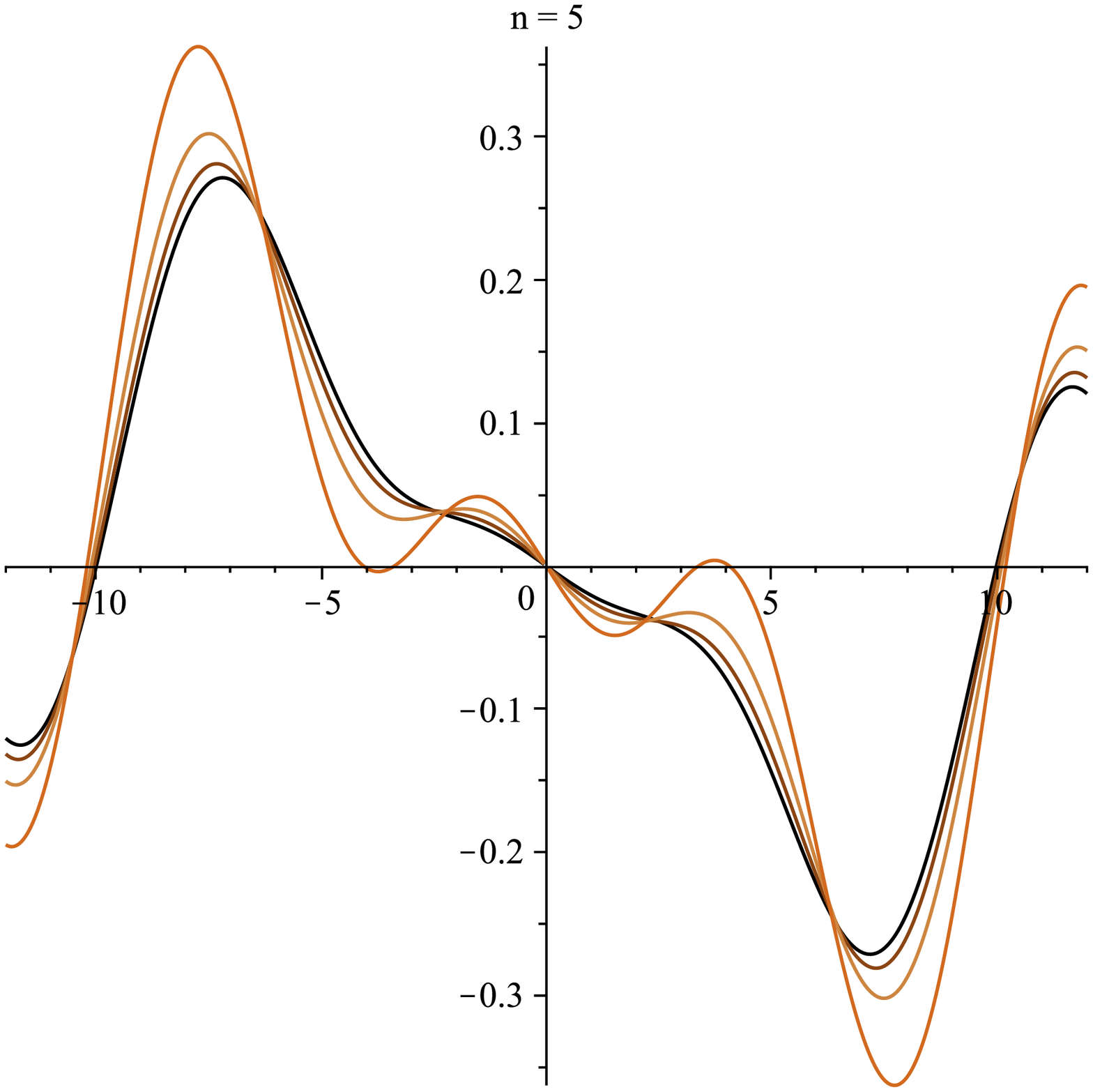}
\caption{The Sobolev--ultrasperical cascade: The first six functions $\varphi_n^{[s]}$ for $s=0,1,2,3$: increasing $s$ corresponds to darker hue.}
\label{fig:3.5}
\end{center}
\end{figure}

In Fig.~\ref{fig:3.5} we display the functions $\varphi_n^{[0]},\ldots,\varphi_n^{[3]}$ for $n=0,1,\ldots,5$. While the convergence to a limit in each figure is quite persuasive, we must beware of `proof by picture': convergence is equally pictorially persuasive in Fig.~\ref{fig:3.4} where, as we have already seen, it need not take place. 

\setcounter{equation}{0}
\setcounter{figure}{0}
\section{Non-symmetric measures}

%
The most obvious non-symmetric weight function is the Laguerre weight $w(\xi)=\ee^{-\xi}\chi_{[0,\infty)}(\xi)$. In that case the $\varphi_n$s are the {\em Malmquist--Takenaka functions,\/} which have a particularly neat form,
\begin{equation}
  \label{eq:MT}
  \varphi_n(x)=\sqrt{\frac{2}{\pi}} \ii^n \frac{(1+2\ii x)^n}{(1-2\ii x)^{n+1}},\qquad n\in\BB{Z}_+,
\end{equation}
\cite{iserles20for}, and they are dense in $\mathcal{PW}_{[0,\infty)}(\BB{R})$. It is possible to extend them to a system dense in all of $\CC{L}_2(\BB{R})$ by melding them with another system, generated by the mirror image of the Laguerre weight, $\ee^{\xi}\chi_{(-\infty,0]}(\xi)$: together we obtain the same system as \R{eq:MT}, except that now $n$ ranges over all of $\BB{Z}$. 

It is, of course, perfectly possible for a system with a non-symmetric measure to be dense in $\CC{L}_2(\BB{R})$, provided that the support of $w$ is all of $\BB{R}$: an example is the Hermite-type weight  $w(\xi)=(1-\xi)^2\ee^{-\xi^2}$. 

\subsection{Shifted Hermite weight}

Letting $\rho\in\BB{R}$, we consider the weight $w(\xi)=\ee^{-(\xi-\rho)^2}$. The underlying orthonormal polynomials are $p_n(\xi)=\tilde{\CC{H}}_n(\xi-\rho)$, where $\tilde{\CC{H}}_n$ is the orthonormalised Hermite polynomial, $\tilde{\CC{H}}_n=\CC{H}_n/\sqrt{2^nn!\sqrt{\pi}}$. Therefore, seeking $\CC{H}^1(\BB{R})$ orthogonality,
\begin{Eqnarray*}
  \varphi_n^{\langle0\rangle}(x,\rho)&=&\frac{\ii^n}{\sqrt{2\pi}} \int_{-\infty}^\infty \tilde{\CC{H}}_n(\xi-\rho) \ee^{-(\xi-\rho)^2/2+\ii x\xi}\D \xi=\ee^{\ii\rho x} \varphi_n^{\langle0\rangle}(x,0),\\
  \varphi_n^{\langle1\rangle}(x,\rho)&=&\frac{\ii^n}{\sqrt{2\pi}} \int_{-\infty}^\infty \tilde{\CC{H}}_n(\xi-\rho) \frac{\ee^{-(\xi-\rho)^2/2+\ii x\xi}}{\sqrt{1+\xi^2}}\D \xi,\qquad n\in\BB{Z}_+.
\end{Eqnarray*}
It is easy to verify that
\begin{displaymath}
  \varphi_n^{\langle0\rangle}(x,\rho)=\ee^{\ii\rho x}\varphi_0^{\langle0\rangle}(x,0),\qquad x,\rho\in\BB{R}.
\end{displaymath}
Thus,  $\varphi_n^{\langle 0\rangle}=\varphi_n^{[0]}$ is merely a complex-valued rotation of the standard Hermite function. The situation is more intriguing with regard to $\varphi_n^{\langle1\rangle}$. Shifting the variable of integration,
\begin{displaymath}
  \varphi_n^{\langle 1\rangle}(x,\rho)=\frac{\ii^n \ee^{\ii\rho x}}{\sqrt{2\pi}} \int_{-\infty}^\infty \tilde{\CC{H}}_n(\xi) \sqrt{\frac{\ee^{-\xi^2}}{1+(\sigma+\xi)^2}} \ee^{\ii x\xi}\D\xi.
\end{displaymath}
On the face of it, we recover an expression similar to \R{eqn:phinformula}, except that $v(\xi)=1+(\sigma+\xi)^2$ is not an even function and does not define a Sobolev inner product. 

\begin{figure}[htb]
\begin{center}
  \includegraphics[width=120pt]{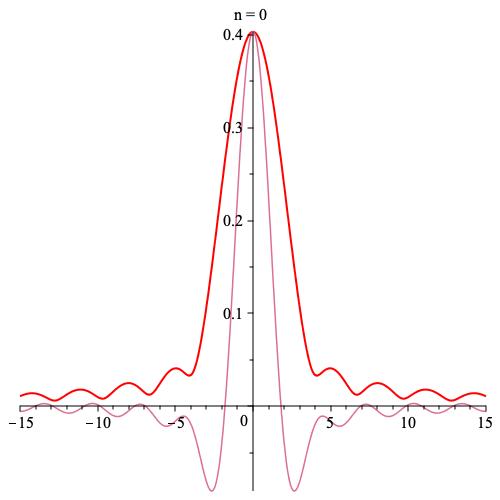}\hspace*{5pt}\includegraphics[width=120pt]{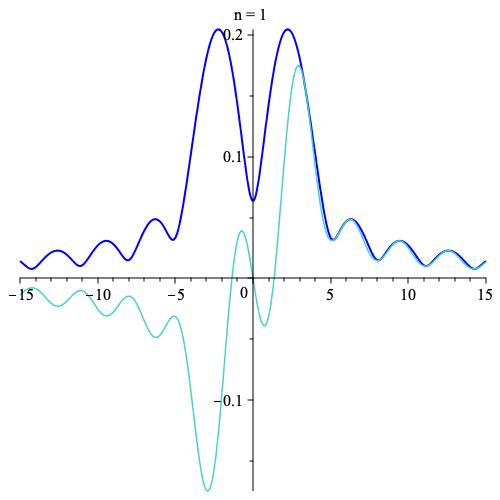}\hspace*{5pt}\includegraphics[width=120pt]{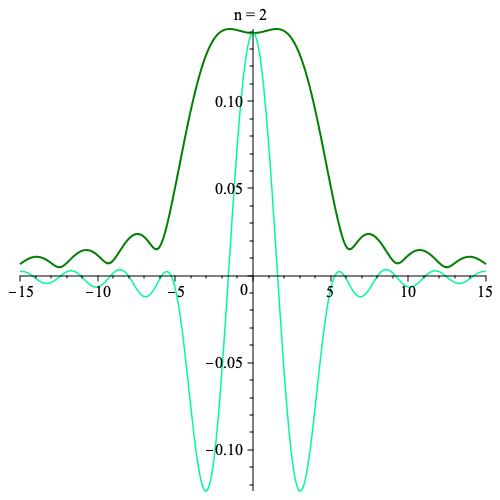}
  \includegraphics[width=120pt]{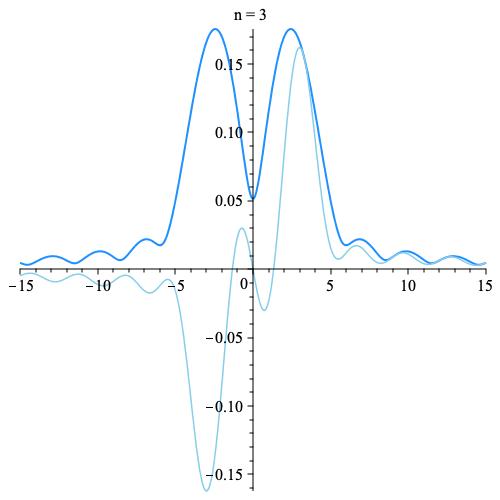}\hspace*{5pt}\includegraphics[width=120pt]{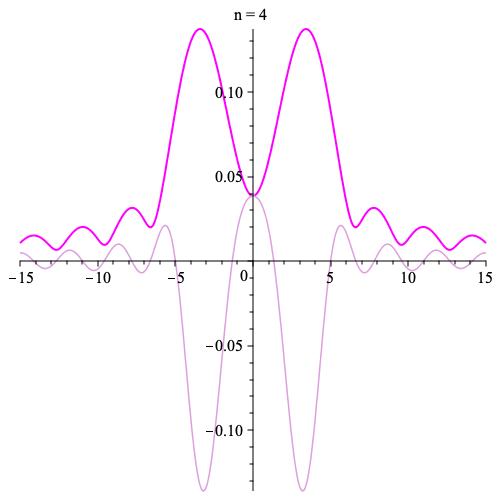}\hspace*{5pt}\includegraphics[width=120pt]{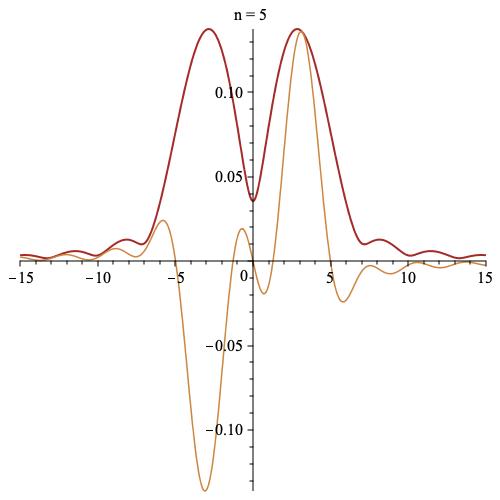}
\caption{Shifted Sobolev--Hermite functions: $|\varphi_n^{\langle 1\rangle}(x,1)|$ in thicker line and darker colour and $\Re\varphi_n^{\langle1\rangle}(x,1)$,  for $n=0,\ldots,5$.}
\label{fig:7.1}
\end{center}
\end{figure}

Fig.~\ref{fig:7.1} displays the absolute and real values of the complex-valued functions $\varphi_n^{\langle1\rangle}$. 

\subsection{The Laguerre weight}

\subsubsection{Sobolev--Laguerre functions of first kind}

Let $w(\xi)=\ee^{-\xi}\chi_{[0,\infty)}(\xi)$, a Laguerre weight. Thus, the $\varphi_n^{\langle0\rangle}$s are Malmquist--Takenaka functions \R{eq:MT}, which we need to complement with their `reflections' for $n\in-\BB{N}$ to form a system dense in $\CC{L}_2(\BB{R})$. By similar token, we need to complement $\varphi_n^{\langle1\rangle}$s, $n\in\BB{Z}_+$, with functions generated with $w(\xi)=\ee^\xi\chi_{(-\infty,0]}(\xi)$ (and indexed by $n\in-\BB{N}$) to attain completeness in $\CC{H}^1(\BB{R})$. 

It is possible to compute individual $\varphi_n^{\langle1\rangle}$s explicitly in terms of Bessel functions of the second kind (a.k.a.\ Weber functions) $\CC{Y}_n$ \cite[10.2.4]{dlmf} and  Struve functions ${\bf H}_n$ \cite[11.2.1]{dlmf}. We first let
\begin{displaymath}
  z=\frac12-\ii x,\qquad g(z)=\CC{Y}_0(z)-{\bf H}_0(z).
\end{displaymath}
The functions $\varphi_n$ can be represented explicitly as linear combinations of derivatives of the function $g$. 

\begin{lemma}
 The explicit form of the functions $\varphi_n^{\langle1\rangle}$ is
 \begin{equation}
   \label{eq:StruveWeber}
   \varphi_n^{\langle1\rangle}(x)=-\frac{\sqrt{2\pi}}{4}\ii^n \sum_{\ell=0}^n {n\choose\ell} \frac{1}{\ell!} g^{(\ell)}\!\!\left(\frac12-\ii x\right)\!,\qquad n\in\BB{Z}_+,
 \end{equation}
 while $\varphi_{-n-1}^{\langle1\rangle}(x)=(-1)^{n+1}\overline{\varphi_n^{\langle1\rangle}(x)}$, $n\in\BB{Z}_+$. 
\end{lemma}

\begin{proof}
 Letting
 \begin{displaymath}
   \eta_n(x)=\frac{1}{\sqrt{2\pi}} \int_0^\infty \xi^n \frac{\ee^{-\xi/2+\ii x\xi}}{\sqrt{1+\xi^2}}\D\xi,\qquad n\in\BB{Z}_+,
 \end{displaymath}
 we compute the generating function
 \begin{Eqnarray*}
    G(x,T)&=&\sum_{n=0}^\infty \frac{\eta_n(x)}{n!} T^n =\frac{1}{\sqrt{2\pi}} \int_0^\infty \frac{1}{\sqrt{1+\xi^2}} \exp\!\left(-\frac{\xi}{2}+T\xi+\ii x\xi\right)\!\D\xi\\
    &=&-\frac{\sqrt{2\pi}}{4}\left[\CC{Y}_0\!\left(\frac12-\ii x-T\right)-{\bf H}_0\!\left(\frac12-\ii x-T\right)\!\right]\!.
 \end{Eqnarray*}
 Therefore
 \begin{displaymath}
   \eta_n(x)=\frac{\partial^n G(x,T)}{\partial T^n}\, \rule[-8pt]{0.75pt}{22pt}_{\,T=0}=(-1)^{n+1}\frac{\sqrt{2\pi}}{4} g^{(n)}\!\!\left(\frac12-\ii x\right)\!,\qquad n\in\BB{Z}_+.
 \end{displaymath}
 
 Laguerre polynomials $\CC{L}_n$ are orthonormal and
 \begin{displaymath}
   \CC{L}_n(x)=\sum_{\ell=0}^n (-1)^\ell {n\choose\ell} \frac{x^\ell}{\ell!},\qquad n\in\BB{Z}_+
 \end{displaymath}
 \cite[18.5.12]{dlmf} and it follows from Theorem~1 that
 \begin{displaymath}
   \varphi_n^{\langle1\rangle}(x)=\frac{\ii^n}{\sqrt{2\pi}} \int_0^\infty \CC{L}_n(\xi) \frac{\ee^{-\xi/2+\ii x\xi}}{\sqrt{1+\xi^2}}\D\xi=\ii^n \sum_{\ell=0}^n (-1)^\ell {n\choose \ell} \frac{\eta_\ell(x)}{\ell!},
 \end{displaymath}
 thereby proving \R{eq:StruveWeber} upon the substitution of the explicit form of $\eta_\ell$. 
 
 Extending this to $n\leq-1$ is trivial.
\end{proof}

\begin{corollary}
  The functions $\varphi_n$ for $n\in\BB{Z}_+$ have the generating function
  \begin{equation}
    \label{eq:GenFunct}
    \sum_{n=0}^\infty \frac{\varphi_n^{\langle1\rangle}(x)}{n!}t^n=-\frac{\sqrt{2\pi}}{4}\ee^{\ii t} \sum_{\ell=0}^\infty \frac{(\ii t)^\ell}{\ell!^2} g^{(\ell)}\!\!\left(\frac12-\ii x\right)\!.
  \end{equation}
\end{corollary}

The proof is elementary, using \R{eq:StruveWeber}. Moreover, \R{eq:GenFunct} can be easily extended to 
\begin{displaymath}
  \sum_{n=-\infty}^\infty\!\frac{\varphi_n^{\langle1\rangle}(x)}{|n|!} \zeta^n=\frac{\sqrt{2\pi}}{4}\!\left[ \ee^{-\ii \zeta^{-1}}\!\sum_{\ell=0}^\infty \frac{(-\ii\zeta^{-1})^\ell}{\ell!^2} g^{(\ell)}\!\left(\frac12\!+\!\ii x\!\right)-\ee^{\ii\zeta} \sum_{\ell=0}^\infty \frac{(\ii\zeta)^\ell}{\ell!^2} g^{(\ell)}\!\left(\frac12\!-\!\ii x\!\right)\!\right]\!,
\end{displaymath}
which makes sense for $|\zeta|=1$.

Since 
\begin{Eqnarray*}
  z^2\CC{Y}_n''(z)+z\CC{Y}_n'(z)+(z^2-n^2)\CC{Y}_n(z)&=&0,\\
  z^2{\bf H}_n''(z)+z{\bf H}_n'(z)+(z^2-n^2){\bf H}_n(z)&=&\frac{z^{n+1}}{2^{n-1}\sqrt{\pi}\CC{\Gamma}(n+\frac12)}
\end{Eqnarray*}
\cite[11.10.5 \& 11.2.7]{dlmf}, it follows that $g$ obeys 
\begin{displaymath}
  z g''(z)+g'(z)+zg(z)=-\frac{2}{\sqrt{\pi}\CC{\Gamma}(\frac12)}=-\frac{2}{\pi}
\end{displaymath}
and we can express $g^{(\ell)}$ as a linear combination of $g$ and $g'$ with rational coefficients. We do not pursue further this course of action. Functions $\varphi_n^{\langle s\rangle}$ for $s\geq2$ (or even the underlying orthogonal polynomials $p_n^{\langle s\rangle}$) are no longer available in an explicit form.

\subsubsection{Sobolev--Laguerre functions of the second kind}

An alternative is to consider the Sobolev--Laguerre cascade of the second kind. While the orthogonal polynomials $p_n^{[s]}$ for the weight $w_s(\xi)=\ee^{-\xi}\chi_{[0,\infty)} \sum_{\ell=0}^s \xi_\ell^2$ are unknown for $s\in\BB{N}$, the underlying moments are trivial to compute and such polynomials can be generated at will. Also the computation of the $\varphi_n^{[s]}$s does not present a problem: for example
\begin{Eqnarray*}
  \varphi_0^{[1]}(x)&=&\sqrt{\frac{2}{3\pi}} \frac{1}{1-2\ii x},\\
  \varphi_1^{[1]}(x)&=&\sqrt{\frac{2}{87\pi}} \ii \left[-\frac{4}{1-2\ii x}+\frac{3(1+2\ii x)}{(1-2\ii x)^2}\right]\!,\\
  \varphi_2^{[1]}(x)&=&\sqrt{\frac{2}{16211\pi}}\ii^2 \left[\frac{34}{1-2\ii x}-\frac{40(1+2\ii x)}{(1-2\ii x)^2} +\frac{29(1+2\ii x)^2}{(1-2\ii x)^3}\right]\!,\\
  \varphi_3^{[1]}(x)&=&\sqrt{\frac{2}{9812127\pi}} \ii^3 \left[-\frac{480}{1-2\ii x}+\frac{762(1+2\ii x)}{(1-2\ii x)^2}-\frac{804(1+2\ii x)^2}{(1-2\ii x)^3}+\frac{559(1+2\ii x)^3}{(1-2\ii x)^4}\right]
\end{Eqnarray*}
and so on: all this seems very similar to \R{eq:MT} and for a good reason: for any $s\in\BB{Z}_+$ we can expand the relevant orthonormal polynomials in the Laguerre basis,
\begin{displaymath}
  p_n^{[s]}(x)=\sum_{j={0}}^n \gamma_{n,j}^{[s]} \CC{L}_j(x)
\end{displaymath}
(cf.\ \R{eq:ConCoeffs}), whereby it follows from \R{eq:second_kind} that
\begin{displaymath}
  \varphi_n^{[s]}(x)=\sum_{j=0}^n  \frac{\gamma_{n,j}^{[s]}}{\sqrt{2\pi}} \int_0^\infty \!\CC{L}_j(\xi) \ee^{-\xi/2+\ii x\xi}\D\xi=\sum_{j=0}^n \gamma_{n,j}^{[s]} \varphi_j^{[0]}(x).
\end{displaymath}
Note that the matrix $\{\gamma_{n,j}^{[s]}\}_{n,j\in\bb{Z}_+}$ is the inverse of the banded connection matrix ${C^{[s]}}^\top$ from Theorem \ref{thm:cascade}.  A similar construction applies also to $\varphi_n^{[s]}$ for $n\leq-1$.

The most remarkable feature of the Malmquist--Takenaka system is that the expansion coefficients $\hat{f}^{[0]}_n=\langle f,\varphi_n^{[0]}\rangle$ can be computed for $-N+1\leq n\leq N$ in $\O{N\log N}$ operations using the Fast Fourier Transform. By \R{eq:BackSubst}, however, once $\hat{\MM{f}}^{[0]}$ is known and assuming that the requisite derivatives of $f$ are available, it costs just $\O{N}$ operations to compute 
\begin{displaymath}
 f_n^{[s]}=\sum_{\ell=0}^s\int_{-\infty}^\infty  f^{(\ell)}(x)\overline{{\varphi_n^{[s]}}^{(\ell)}(x)}\D x,\qquad -N+1\leq n\leq N.
\end{displaymath}
(The derivatives of $\varphi_n^{[s]}$s can be computed similarly to the functions themselves, ${\MM{\varphi}_n^{[s]}}^{(\ell)}={C^{[s]}}^{-\top} {\MM{\varphi}^{[0]}}^{(\ell)}$.) Altogether, the cost scales as $\O{sN\log N}$.

\setcounter{equation}{0}
\setcounter{figure}{0}
\section{Conclusion}

In a sequence of previous papers \cite{iserles19oss,iserles20for,iserles21fco,Iserles21daf} the current authors have sketched different aspects of an overarching theory of $\CC{L}_2$-orthonormal systems on the real line with a tridiagonal differentiation matrix. In this paper we extend the framework to orthogonality with respect to Sobolev spaces. Unlike in the case of orthogonal polynomials, where Sobolev orthogonality is of a completely different flavour to orthogonality with respect to a Borel measure \cite{iserles91opo,marcellan91ops,marcellan15oso}, in our case we can leverage many elements of the ``$\CC{L}_2$ theory'' to a Sobolev setting: the connection to standard orthogonal polynomials via a weighed Fourier transform, density in Paley--Wiener spaces and fast computation of certain expansions. Other aspects of the theory are new, in particular the existence of two cascades, the latter of which can be ascended by banded triangular connection coefficients.

The work of this paper is a stepping stone toward the development of spectral methods on the real line that respect a wide range of invariants that can be expressed as conservation of a variable-weight Sobolev norm: a couple of examples have been given in Section~1. We expect to return to this issue in a forthcoming paper.

\section*{Acknowedgements}

The authors are grateful for very useful and enlightening correspondence with Enno Diekma, Erik Koelink and Tom Koornwinder. We gratefully acknowledge the partial support by the Simons Foundation Award No 663281 granted to the Institute of Mathematics of the Polish Academy of Sciences for the years 2021--2023. MW ackowledges support by Computational Mathematics in Quantum Mechanics, Grant of the National Science Centre (SONATA-BIS), project no. 2019/34/E/ST1/00390.

\bibliographystyle{agsm}
\bibliography{arXivVersion}

\end{document}